\documentclass[12pt]{amsart} 
\usepackage{amsmath,amsthm,amssymb,graphicx,eucal, amscd, ifthen} 
\usepackage[margin=1in]{geometry} 
\usepackage{hyperref}

\numberwithin{equation}{section} 
\numberwithin{figure}{section} 

 \newtheorem{theorem}{Theorem}[section]
 \newtheorem{definition}[theorem]{Definition}
 
 \newtheorem{conjecture}[theorem]{Conjecture}

 \newtheorem{proposition}[theorem]{Proposition}
 \newtheorem{lemma}[theorem]{Lemma}
 \newtheorem{corollary}[theorem]{Corollary}
 
 \newtheorem{remark}[theorem]{Remark}
 
 \newtheorem{example}[theorem]{Example}
 
\def\Sig{Sierpi\'nski gasket}
\def\Si{Sierpi\'nski}
\def\Sif{Sierpi\'nski fractafold}

\def\frSiff{finitely ramified Sierpi\'nski fractal field}
\def\iSif{infinite Sierpi\'nski fractafold}
\def\iSig{infinite Sierpi\'nski gasket}
\def\arti{and references therein}
\def\pcf{\mbox{p.c.f.}}
\def\iFF{if and only if }
\def\Lp{Laplacian}

 \def\pvz{\psi_{v,\lam}}
 \def\puz{\psi_{u,\lam}}
 \def\Pz{\Psi_\lam}
 \def\fR{\mathfrak R}
 \def\sR{\mathcal R}
 \def\fF{\mathfrak F}
 
 \def\sE{\mathcal E}
 \def\sJ{\mathcal J}
 
 \def\lam{\lambda}

\long\def\BEQ#1#2{\begin{equation}\label{#1}#2\end{equation}}
\long\def\BB#1#2#3{\begin{#1}\label{#2}#3\end{#1}}

\def\DS{\Delta_{SG}}
\def\D{\Delta}
\def\G{\Gamma}
\def\Go{{\Gamma_0}}

\def\Sum{\displaystyle\sum}

\begin{document}

   \author{Robert S. Strichartz}
   \address{Department of Mathematics\\
Cornell University\\
Ithaca, NY 14853}
   \email{str@math.cornell.edu}

   \author{Alexander Teplyaev}
   \address{Department of Mathematics\\
University of Connecticut\\
Storrs, CT 06269-3009
}
   \email{teplyaev@math.uconn.edu}
      \thanks{Research supported in part by the National Science Foundation, grants DMS-0652440 (first author) and DMS-0505622 (second author).}

   \title{Spectral analysis on infinite 
   Sierpinski  
   fractafolds}

   \begin{abstract}
A fractafold, a space that is locally modeled on a specified fractal, is the fractal equivalent of a manifold.  For compact fractafolds based on the \Si\ gasket, it was shown by the first author how to compute the discrete spectrum of the Laplacian in terms of the spectrum of a finite graph \Lp.  A similar problem was solved by the second author for the case of infinite blowups of a \Sig, where spectrum is pure point of infinite multiplicity.  Both works used the method of spectral decimations to obtain explicit description of the eigenvalues and eigenfunctions.
In this paper we combine the ideas from these earlier works  to obtain a description of the spectral resolution of the Laplacian for noncompact fractafolds. Our   main abstract results    enable us to obtain a completely explicit description of the spectral resolution of the fractafold Laplacian. For some specific examples we   turn the spectral resolution into a ``Plancherel formula''. We also present such a formula  for the graph Laplacian on the 3-regular tree, which appears to be a new result of independent interest.  
In the end we  discuss periodic fractafolds and fractal fields.
 \tableofcontents  \end{abstract}


   \date{\today}

   \maketitle
   
\section{Introduction}

Analysis on fractals has been developed based on the construction of Laplacians on certain basic fractals, such as the \Si\ gasket, the Vicsek set, the \Si\ carpet, etc., which may be regarded as generalizations of the unit interval, in that they are both compact and have nonempty boundary.  As is well-known in classical analysis, it is often more interesting and sometimes simpler to deal with spaces like the circle and the line, which have no boundary, and need not be compact.  The theory of analysis on manifolds is the natural context for such investigations.  The notion of \emph{fractafold}, introduced in \cite{St03}, is simply the fractal equivalent: a space that is locally modeled on a specified fractal.  For compact fractafolds based on the \Si\ gasket, it was shown in \cite{St03} how to compute the spectrum of the Laplacian in terms of the spectrum of a \Lp\ on a graph $\G$ that describes how copies of SG are glued together to make the fractafold.  On the other hand, in \cite{T98} a similar problem was solved for the case of infinite blowups of SG.  These are noncompact fractafolds where the graph $\G$ mirrors the self-similar structure of SG.  Not surprisingly, the spectrum in the compact case is discrete, and the eigenvalues and eigenfunctions are described by the method of spectral decimation introduced in \cite{FS}.
The surprise is that for the infinite blowups the spectrum is pure point, meaning that there is a basis of $L^2$ eigenfunctions (in fact compactly supported), but each eigenspace is infinite dimensional and the closure of the set of eigenvalues is a Cantor set.  Again the method of spectral decimations allows an explicit description of the eigenvalues and eigenfunctions.

In this paper we combine the ideas from these earlier works \cite{St03,T98} to obtain a description of the spectral resolution of the Laplacian for noncompact fractafolds with infinite cell graphs $\G$.  The graph $\G$ is assumed to be 3-regular, so the fractafold has no boundary.  The edge graph $\Go$ is then 4-regular, and the fractafold is obtained as a limit of graphs obtained inductively from $\Go$ by filling in detail (that is, each graph triangle is eventually replaced with a copy of the \Sig).  Our first main abstract result is Theorem~\ref{thm-main}, which describes how to obtain the spectral resolution of the Laplacian on the fractafold from the spectral resolution of the graph Laplacian on $\Go$.  This is a version of spectral 
 decimation, and uses an idea from \cite{OSSt} to control the $L^2$ norms of functions under spectral decimation.  The second main abstract result is Theorem \ref{thm4.1}, which shows how to obtain the spectral resolution of the graph Laplacian on $\Go$ from the spectral resolution of the graph Laplacian on $\G$ using ideas from \cite{Shirai,St10}.  We note that the spectral resolution on $\Go$ may or may not contain the discrete eigenvalues equal to 6, and the explicit determination of the 6-eigenspace and its eigenprojector must be determined in a case-by-case manner.  Combining the two theorems enables us to obtain a completely explicit description of the spectral resolution of the fractafold Laplacian to the extent that we are able to solve the following problems:
\begin{itemize}
\item[(a)] Find the explicit spectral resolution of the graph Laplacian on $\G$;
\item[(b)] Find an explicit description of the 6-eigenspace and its eigenprojector for the graph Laplacian on $\Go$.
\end{itemize}

The bulk of this paper is devoted to solving these two problems for some specific examples.  However, we would like to highlight another problem that arises if we wish to turn a spectral resolution into a ``Plancherel formula''.  Typically we will write our spectral resolutions as
\begin{equation}\label{1.1}
f(x)=\int_{\sigma(-\Delta)} \left(\int P(\lambda, x, y) f(y) d\mu(y) \right) dm(\lambda) \end{equation} 
where $P(\lambda, x, y)$ is an explicit kernel realizing the projection onto the $\lambda$-eigenspace, i.e.
\begin{equation}\label{1.2}
-\Delta\int P(\lambda, x, y) f(y) d\mu(y) = \lambda \int P(\lambda, x, y) f(y)d\mu(y) \end{equation}
and $dm(\lambda)$ is a scalar spectral measure. (Here neither $P(\lambda, x, y)$ nor $dm(\lambda)$ are uniquely determined, since we can clearly multiply them by reciprocal functions of $\lambda$ while preserving \eqref{1.1} and \eqref{1.2}.)  If we write 
\begin{equation}\label{1.3}
P_\lambda f(x)=\int  P(\lambda, x, y) f(y)d\mu(y) \end{equation}
then \eqref{1.1} resolves $f$ into its components $P_\lambda f$ in the $\lambda$-eigenspaces.  A Plancherel formula would express the squared $L^2$-norm $||f||^2_2$ in terms of an integral of contributions from the components $P_\lambda f$.  In the case of pure point spectrum this is straightforward, for then the $\lambda$-integral is a discrete possibly infinite sum, and we just have to take the $L^2$-norm of each $P_\lambda f$, so 
\begin{equation}\label{1.4}
||f||_2^2 =\Sum_{\lambda \in \sigma(-\Delta)} ||P_\lambda f||_2^2  \end{equation}
where $P_\lambda$ is the eigenprojection. The spectral measure  $m$ is the counting measure in this case. 

In the case of a continuous spectrum this is decidedly not correct, and there does not appear to be a generic method to obtain the correct analog.  So we pose this as a third problem:
\begin{itemize}
\item[(c)] describe explicitly a Hilbert space of $\lambda$-eigenfunctions with norm $||\phantom{Pf}||_\lambda$ such that $||P_\lambda f||_\lambda$ is finite for m-a.e. $\lambda$ and 
\begin{equation}\label{1.5}
||f||_2^2 = \int_{\sigma(-\Delta)} ||P_\lambda f||_\lambda ^2 dm(\lambda). \end{equation} \end{itemize}
This problem is interesting essentially only when the eigenspace is infinite dimensional. 
The resolution of this problem in some classical settings is discussed in \cite{St89} and \cite{AIonescu}.  Here we present a solution to this problem for the graph Laplacian on the 3-regular tree.  This result appears to be new, and is of independent interest.  

The first specific examples we consider is the \emph{tree fractafold}, discussed in Section \ref{sec5}, where $\G$ is the 3-regular tree.  In this case the solution to a) is well-known \cite{Cartier,F-TN}.  We solve (b) by showing that the 6-eigenspace on $\Go$ is infinite dimensional and we give an explicit tight frame for this space.  We solve (c) in terms of a mean value on the tree that is in fact different from the obvious mean value.  The fractafold spectrum in this example is a union of point spectrum and absolutely continuous spectrum.

In Section \ref{sec6} we discuss periodic fractafolds, concentrating on a \emph{honeycomb fractafold} where $\G$ is a hexagonal lattice.  In this case the solution to a) is also well-known.  Our solution to b) gives a basis for the infinite dimensional 6-eigenspace of compactly supported functions.  Finally in Section \ref{sec-frSiff} we discuss an example of a finitely ramified periodic \Si\ fractal field  (see \cite{HK2}) that is not a fractafold, but can be treated using our methods.

Essentially all the results of this paper can be extended to fractafolds based on the $n$-dimensional \Sig, using similar methods. It seems likely that similar results could be obtained for any \pcf\ fractal for which the method of spectral decimation applies (see 
\cite[\arti]
{eigenpapers,DS,FordSteinhurst,FS,HSTZ,Jordan,Ki01book,KT,MT03,Sh,St06book,Zhou1,Zhou2}).

 \subsection*{Acknowledgments} We are  grateful to 
Peter Kuchment and Daniel Lenz for very helpful discussions, and to Eugene~B.~Dynkin for asking questions about the periodic fractal structures. We thank Matthew Begue for help in the manuscript preparation.

\section{Set-up results for \iSif s}\label{sec-iSif}

\subsection{\Lp\ on the \Sig} 
\label{subsec-lp-}
We denote by $\DS$ the standard \Lp\ on $SG$, and by $\mu_{SG}$ the standard normalized Hausdorff probability measure on ${SG}$ (see \cite{Ki89,Ki01book,St06book} for details). 
\begin{figure}[htb]\centering
\includegraphics[height=26pt,width=30pt]{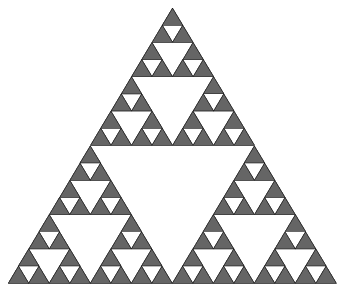}
\caption{\Sig.}\label{fig-Sig}
\end{figure}
The \Lp\ $\DS$ is self-adjoint on $L^2({SG},\mu_{SG})$ with appropriate boundary conditions (usually Dirichlet or Neumann).  
The \Lp\ $\DS$ can be defined either probabilistically or analytically, using Kigami's resistance (or energy) form and the relation
$$\mathcal E(f,f)=-\frac32\int_{SG} f\DS fd\mu_{SG}$$
for functions in the corresponding domain of the \Lp. The  energy is defined by 
$$\mathcal E(f,f)=\lim_{n\to\infty}\left(\frac53\right)^n\sum_{x,y\in V_n, x\sim y}\left(f(x)-f(y)\right)^2.$$ In these formulas $V_n$ is a finite set of $(3^{n+1}+3)/2$ points 
in $SG$ that are at the Euclidean distance $2^{-n}$ from the neighboring points, and $\sim$ denotes the recursively defined graph structure on $V_n$. Note the normalization factor $\frac32$ that is inserted here for the convenience of computation (see \cite{St06book} the explanations). 

\subsection{Spectral decimation and the eigenfunction extension map} 
\label{subsec-s-d}
Both Dirichlet and Neumann spectra of $\DS$ are well known (see \cite{FS,St06book,T98}).  
To compute the spectrum of $\DS$ 
one employs  the so called spectral decimation method  using inverse iterations of the polynomial $$R(z)=z(5-z).$$ By convention the eigenvalue equation is written $-\DS u = \lam u$ because $-\DS$ is a non-negative 
operator. Each positive eigenvalue can 
be written as 
\BEQ{e-lam-}{\lam=       \lim_{m\to\infty}5^m\lam_m=       5^{m_0}\lim_{k\to\infty}5^{k}\lam_{k+{m_0}}} for a sequence $\{\lam_m\}_{m=m_0}^\infty$
 such that $\lam_m=R(\lam_{m+1})$ and $\lam_{m_0}\in\{2,5,6\}$, which can be written as $$R^{\circ k}(\lam_{k+m_0})\in\{2,5,6\},$$
 where the powers $R^{\circ k}$ of $R$ are composition powers.
If we denote $\fR_k(z)=R^{\circ k}(5^{-k}z)$ then 
\BEQ{e-Rf-}{
R^{\circ k}(\lam_{k+m_0})=\fR_k(5^k \lam_{k+m_0})=\fR_k\left(\frac235^{-m_0}
       5^{m_0} 5^k\lam_{k+m_0}\right)}
Thus an important role is played by the function \BEQ{e1}{\fR(z)=\lim_{k\to\infty}R^{\circ k}(5^{-k}z).}  This is an analytic function, which is a classical object in complex dynamics, and a recent detailed study and background can be found in \cite{DGV,DGV2}.
In the context of the \Lp\ on the \Sig\ this function first appeared in 
\cite[Lemma 2.1]{ORS2009} and \cite[Remark 2.5]{DRS2009} (see also \cite{IPRRS,ORS2010} for related results). 
  In particular, this function can be defined as the solution   of the classical functional equation \BEQ{e1f}{R(\fR(z))=\fR(5z).} 
Note that, in a neighborhood of zero,  the inverse of the function $\fR$ can be defined by 
 \BEQ{e1i}{\sR(w)=\lim_{k\to\infty}5^{k}R^{-k}(w),} and satisfies the 
 functional equations \BEQ{e1if}{5\sR(w)=\sR(R(w)),} in a  neighborhood of zero.
 
 One can see from \eqref{e-Rf-} that each nonzero eigenvalue $\lam$ satisfies 
 $$\lam\in       5^{m_0}\fR^{-1}\{2,5,6\}\subset       \bigcup_{m=0}^\infty
 5^{m}\fR^{-1}\{2,5,6\}.$$ Some of the points in the union of these 
 sets are so-called ``forbidden eigenvalues'', and the rest are so-called  2-series, 5-series and 6-series eigenvalues (see \cite{St06book}). 
A detailed analysis shows that 
the spectrum of the Dirichlet \Lp\ is  $$\Sigma_{D}=       5\left(\fR^{-1}\{2,5\}\cup5\fR^{-1}\{5\}\bigcup_{{m_0}=2}^\infty5^{m_0}\fR^{-1}\{3,5\}\right)$$ 
 and the spectrum of the 
Neumann \Lp\ is 
 $$\Sigma_{N}=\{0\} \cup       5 \left(\fR^{-1}\{3\}\cup\bigcup_{{m_0}=1}^\infty5^{m_0}\fR^{-1}\{3,5\}\right).$$ The multiplicities, which grow exponentially fast with $k$, were computed explicitly in \cite{FS}, and also can be found in \cite{eigenpapers,St06book,T98}. 
  Note that, because of the functional equations \eqref{e1f} and \eqref{e1if}, and because $R(2)=R(3)=6$, we have $$5\left(\fR^{-1}\{2\}\cup\fR^{-1}\{3\}\right)=\fR^{-1}\{6\}.  $$

If we define 
$$
\Sigma_{ext}=        5 \left(\fR^{-1}\{2\}\cup
\bigcup_{{m}=0}^\infty5^{m}\fR^{-1}\{5\}\right)\subset\fR^{-1}\{0,6\}.
$$
then we have the following proposition. 
\BB{proposition}{aa}{For any $v\in \partial {SG}$ and any complex number $\lam\notin \Sigma_{ext}$ there is a unique continuous function $\pvz(\cdot):{SG}\to\mathbb R$, such that  $\pvz(v)=1$, $\pvz$ vanishes at the other two boundary points, and the pointwise eigenfunction equation $-\D\pvz(x)=\lam\pvz(x)$ holds at every point $x\in {SG}\backslash \partial {SG}$.} 
Naturally, $\pvz$ is called the eigenfunction extension map, which is explained in \cite[Section 3.2]{St06book}, and the proposition is essentially the same as \cite[Theorem 3.2.2]{St06book}. 


\BB{example}{ex-01}{{ Spectral decimation for the unit interval [0,1].} \normalfont In order to illustrate these notions we briefly explain how they look 
 in a more classical case of the unit interval. We have that 
 $\D_{[0,1]}=\frac {d^2}{dx^2}$ is the standard \Lp\ on ${[0,1]}$, and if $\mu_{{[0,1]}}$ is the Lebesgue measure on ${{[0,1]}}$ then $\D_{[0,1]}$ is self-adjoint  and $$\mathcal E(f,f)=\int_0^1 (f'(x))^2dx=-\int_{{[0,1]}} f\D_{[0,1]} fd\mu_{{[0,1]}}$$
for functions in the   domain of the Dirichlet or Neumann \Lp. The  energy can also be  defined by 
$\mathcal E(f,f)=\lim_{n\to\infty}2^n\sum_{x,y\in V_n, x\sim y}\left(f(x)-f(y)\right)^2$ where $V_n=\{k/2^n\}_{k=0}^{2^n}$. 
To compute the spectrum of $-\D_{[0,1]}$ 
one can use  the  spectral decimation method  with inverse iterations of the polynomial $R (z)=z(4-z).$ Each positive eigenvalue can 
be written as 
$\lam=       \lim_{m\to\infty}4^m\lam_m$ 
for a sequence $\{\lam_m\}_{m=m_0}^\infty$
 such that $\lam_m=R(\lam_{m+1})$ and $\lam_{m_0}\in\{0,4\}$.
Then $\fR(z)=\lim_{k\to\infty}R^{\circ k}(4^{-k}z)=2-2\cos(\sqrt z)$ satisfies the 
functional equation $R(\fR(z))=\fR(4z)$. In this case
$\sigma({-\D_{[0,1]}})\subset \fR^{-1}\{0,4\}$, the multiplicity is one,  
and 0 is in the  Neumann spectrum but not in the Dirichlet spectrum. 
The eigenfunction extension map is $$\pvz(x)=\cos(\sqrt\lam\, |x{-}v|)-\frac{\cos(\sqrt\lam)}{\sin(\sqrt\lam)}\sin(\sqrt\lam \,|x{-}v|)$$ where $v$ is 0 or 1. 

For much more information on this example and its relation to quantum graphs see \cite{Post2008} and references therein. }

\subsection{Underlying graph assumptions and \Sif s} 
Let $\Go=(V_0,E_0)$ be a finite or  infinite graph. To define a \Sif, we  assume that $\Go$ is a 4-regular graph which is a union of complete graphs of  3 vertices. It can be said that $\Go$ is a regular 3-hyper-graph in which every vertex belongs to two hyper-edges. A hyper-edge in this case is a complete graphs of 3 vertices, and we call it a cell, or a 0-cell, of $\Go$. We denote the discrete \Lp\ on $\Go$ by $\D_\Go$. (In principle, these  assumptions can be weakened; see Section~\ref{sec-frSiff} and Figure~\ref{fig-t-lattice} for instance). 

Let ${SG}$ be the usual compact \Sig\ (see Figure~\ref{fig-iSig}). We define a \Sif\  $\fF$ by replacing each cell of $\Go$ by a copy of ${SG}$. These copies we call  cells, or 0-cells, of the \Sif\  $\fF$. 
Naturally, the corners of  the copies of the \Sig\ ${SG}$ are identified with the vertices of $\Go$. 

A fractafold is called infinite if the graph $\Go$ is infinite. In particular, finite fractafolds are compact and infinite fractafolds are not compact. All the details  can be found in  \cite{St03}. In this paper we use the same notation as in \cite{St03} as much as possible (see also  \cite{St10}). 
Since the pairwise intersections of the 
cells of the \Sif\  $\fF$ are finite, we can consider the natural measure on the \Sif\  $\fF$, which we also  denote $\mu$. Furthermore, since $\DS$ is a local operator, we can define a local  \Lp\ $\D$ on the \Sif\  $\fF$,   as explained in \cite{St03}.

 \subsection{Eigenfunction extension map on fractafolds} \label{subsec-eem}
 
 For any $v\in V_0$ and $\lam\notin\Sigma_{ext}$ there is a unique continuous function $\pvz(\cdot):\fF\to\mathbb R$ such that \begin{enumerate}\item the support of $\pvz$ is contained in the union of of the cells of the \Sif\  $\fF$ that contain $v$;\item $\pvz(v)=1$;\item the pointwise eigenfunction equation $$-\D\pvz(x)=\lam\pvz(x)$$ holds at any point $x\in \fF\backslash V_0$.
\end{enumerate}
For any function $f_0$ on $\Go$ (and any $\lam$ as above), we define the eigenfunction extension map by  \BEQ{e2}{\Psi_{\lam}  f_0(x)=\sum_{v\in V_0}f_0(v)\pvz(x).} By definition, $f=\Psi_\lam f_0$ is a continuous extension of $f_0$ to the \Sif\  $\fF$ which is a pointwise solution to the eigenvalue equation above for all $x\in \fF\backslash V_0$. Moreover, 
it is known that if $f_0$ is a pointwise solution to the eigenfunction equation $-\D_\Go f_0=\lambda_0f_0$ on $\Go$, and $\lambda_0\notin\{0,6\}$, then $f=\Psi_{\lam} f_0$ is a continuous extension of $f_0$ to the \Sif\  $\fF$ which is a pointwise solution to the eigenvalue equation above for all $x\in \fF$. Note that here we have   $\lam\in       \fR^{-1}(\lambda_0)$, where $\mathfrak  R$ is as above. The eigenfunction extension map is explained in \cite{St06book} on page 69. 

It is easy to see that $\Pz:\ell^2(V_0)\to L^2(\fF,\mu)$ is a bounded linear operator for any $\lam\notin    \fR^{-1}\{2,5,6\}$, and its adjoint 
$\Pz^*:L^2(\fF,\mu)\to\ell^2(V_0) $ 
can be computed as \BEQ{e3}{\Big(\Pz^* g\Big)(v)=\int_\fF g(x)\pvz(x)d\mu(x).} 

\subsection{Spectral decomposition (resolution of the identity)} We suppose that the self-adjoint discrete \Lp\ $\D_\Go$ on $\Go$ has a spectral decomposition (resolution of the identity) \BEQ{e4}{-\D_\Go=\int_{\sigma(-\D_\Go)}\lambda dE_\Go(\lambda). } 
 which has a form 
 \BEQ{e7}{-\D_\Go f_0(v)=\int_{\sigma(-\D_\Go)}\lambda\sum_{u\in V_0} P_\Go(\lam,u,v)f_0(u)dm_\Go(\lam) }
 where $m(\cdot)$ is a spectral measure of $-\D$ which is a Borel measure on $\sigma(-\D_\Go)$ (see Section~\ref{sec-general} for more detail).
 
We define a function $M(\lam ) $  as the infinite product
\BEQ{eM}{M(\lam) =\prod_{m=1}^\infty 
\frac{(1-\frac15\lam_m)(1-\frac12\lam_m)}{(1-\frac16\lam_m)(1-\frac25\lam_m)} }
where 
$$\lam=       \lim_{m\to\infty}5^m\lam_m$$ and $\lam_m=R(\lam_{m+1})$.
The function $M(\cdot )$ is  known from \cite[Lemma 2.2 and Corollary 2.4]{OSSt} where it  appears when the $L^2$ norm of eigenfunctions on the \Sig\ is computed. This function   does not depend on the fractafold, but only on the \Sig.

We denote $$\Sigma_\infty=       5\left(\fR^{-1}\{2\}\cup\bigcup_{m=0}^\infty
 5^{m}\fR^{-1}\{3,5\}\right)$$
 and
$$\Sigma_\infty'=       5\left(\bigcup_{m=1}^\infty
 5^{m}\fR^{-1}\{3,5\}\right)\subset\Sigma_\infty.$$
 Note that for the difference of these two sets we have 
 $$\Sigma_\infty\backslash\Sigma_\infty'=       5\fR^{-1}\{2,3,5\}
  \subset\fR^{-1}\{0,6\}.$$

\begin{theorem}\label{thm-main}
The \Lp\ $\D$ is self-adjoint and \BEQ{e5}{\fR^{-1}(\sigma(-\D_\Go))\,\cup\,\Sigma_\infty'
\subset
\sigma(-\D)\subset \fR^{-1}(\sigma(-\D_\Go))\,\cup\,\Sigma_\infty.} Moreover, the spectral 
 decomposition  $$-\D=\int\limits_{\sigma(-\D)}\lam  dE(\lam ) $$ can be written as 
 \BEQ{e6}{-\D=       
 \int\limits_{\fR^{-1}(\sigma(-\D_\Go))\backslash\Sigma_\infty}\lam  M(\lam ) \Pz^*d\Big(E_\Go(\fR(\lam))\Big)\Pz+\sum_{\lam \in\Sigma_\infty}\lam \,E\{\lam \}.}
 Here 
 $E\{\lam \}$ denotes the eigenprojection if $\lam $ is an eigenvalue (the eigenprojection is non-zero \iFF $\lam $ is  an eigenvalue). 
 
 All eigenvalues and eigenfunctions of  $\D$  can be computed by the spectral decimation method as so called offspring of either localized eigenfunctions on approximating graph \Lp s, or of eigenfunctions on $\Go$.  
Furthermore, the \Lp\ $\D$ on the \Sif\  $\fF$ has 
 the spectral decomposition of the   form 
 \BEQ{e8}{-\D f(x)=       
 \int\limits_{\fR^{-1}(\sigma(-\D_\Go))\backslash\Sigma_\infty}\lam \,\left(\int_\fF P(\lam ,x,y)f(y)d\mu(y)\right)dm(\lam )\, +\sum_{\lam \in\Sigma_\infty}\lam \,E\{\lam \}f(x)} 
 where $m=m_\Go\circ\fR$  
 and 
\BEQ{e9}{ P(\lam ,x,y)=M(\lam )\sum_{u,v\in V_0}\pvz (x)\puz (y)P_\Go(\fR(\lam ),u,v).} 
 \end{theorem}

\begin{proof} 
Let $\G_0=(V_0,E_0)$ be as above and let $\G_1=(V_1,E_1)$ be a  graph obtained from $\G_0$ by replacing each cell of $\G_0$ with the graph shown below. \newline
\def\trigup{\put(0, 0){\line(1, 0){6}}\put(0, 0){\line(3, 5){3}}\put(6, 0){\line(-3, 5){3}}}
\def\trigdown{\put(0, 0){\line(1, 0){6}}\put(0, 0){\line(3, -5){3}}\put(6, 0){\line(-3, -5){3}}}
\centerline{
\begin{picture}(30, 35)(0, -5)
\setlength{\unitlength}{2.5pt}\thicklines
\put(0,0){\trigup}
\put(3,5){\trigup}
\put(6,0){\trigup}
\end{picture}}
\newline
\noindent The three vertices of the biggest triangle in the above graph replace  the three vertices  of each cell of $\G_0$.  We repeat this procedure recursively to define a sequence of discrete approximations $V_n$ to the fractafold the \Sif\ $\fF$. On each 
$V_n$ we consider discrete energy form, which converge as $n\to\infty$ 
with the same normalization as in Subsection~\ref{subsec-lp-}. 
In the limit we obtain a resistance form $\sE$ of the \Sif\  $\fF$ and 
 one can use the theory of resistance forms of Kigami (see \cite{Ki01book,Ki03})  to define the weak \Lp\ $\D$ on the \Sif\  $\fF$.  
 More precisely, the resistance   form is a regular Dirichlet form on $L^2(\fF,\mu)$ by \cite[Theorem 8.10]{Ki03}, for which a 
self-adjoint \Lp\ $\D$ is uniquely defined (see \cite[Proposition 8.11]{Ki03}). One can easily see that in this case the set of continuous compactly supported functions in $Dom\D$ such that $\D f$ is also continuous (and also compactly supported) form a core. For any such function $f$ the \Lp\ $\D f$ can be approximated by discrete 
\Lp s, that is $\D f(x)=\lim_{n\to\infty}5^n\D_n f(x)$, where $\D_n$ is the graph \Lp\ on $V_n$. The limit is pointwise for each $x\in V_*=\bigcup V_n$, and is unform on compact subsets of the \Sif\  $\fF$ provided $\D f$ is continuous with compact support. The pointwise  and uniform convergence of discrete \Lp s in this case is justified in the same as way in the case of the \Lp\ on the \Sig. 

Using notation of Subsections~\ref{subsec-lp-} and \ref{subsec-s-d}, we denote $m_n=m_\Go\circ\fR_n$  and 
$$ P_n(\lam ,x,y)=M_n(\lam )\sum_{u,v\in V_0}\pvz (x)\puz (y)P_\Go(\fR(\lam ),u,v)$$ 
where  $M_n(\lam )$ is defined as the partial product in the definition of $M(\lam )$. We further denote $$\Sigma_n=       5\left(\fR^{-1}_n\{2\}\cup\bigcup_{m=0}^{n-1}
 5^{m}\fR^{-1}_m\{3,5\}\right)$$ and let $E_n{\lam}$ be the eigenprojection of $-\D_n$ 
 corresponding to $\lam$.
 Then we have the discrete version of the formula \eqref{e8} because of the computation in \cite[Theorem 3.3]{eigenpapers}  (see also Sections \ref{sec-general} and \ref{sec5} below, where $P_\Go(\lam ,u,v)$ is denoted by $\tilde{P_\lambda}(u,v)$). Note that in \cite[Theorem 3.3]{eigenpapers} the normalization factor was $1/(\phi R')$, where $\phi(z)=\frac{3-2z}{(5-4z)(1-2z)}$ and $R(z)=z(5-4z)$, which produces the normalization factor 
$$\frac{(5-4z)(1-2z)}{(3-2z)(5-8z)}=\frac13\frac{(1-4z/5)(1-4z/2)}{(1-4z/6)(1-8z/5)},$$
which is the same as in \eqref{eM}. Here $4z$ replaces $\lam_m$ because the distinction between probabilistic and graph \Lp s, and the extra factor 
 $\frac13$ is because of the integration in \eqref{e8}. 

Let $u$ and $f$ be continuous functions on the \Sif\  $\fF$ with compact support and let 
$$v=(-\D+1)^{-1}f.$$ The usual energy and $L^2$ estimates imply that $v\in Dom(\D)$ is continuous, square integrable, and 
$-\D v=f-v$. We have, by the discrete approximations,  that  the inner product 
$\langle u,v\rangle_{L^2}$ is equal to $$
\int\limits_{\fR^{-1}(\sigma(-\D_\Go))\backslash\Sigma_\infty}\frac1{\lam+1} \,\left\langle u,\int_\fF P(\lam ,x,y)f(y)d\mu(y)\right\rangle_{L^2}dm(\lam )\, +\sum_{\lam \in\Sigma_\infty}\frac1{\lam+1} \,\left\langle u,E\{\lam \}f\right\rangle_{L^2}
$$ 
and so we have the relation 
  $$\langle u,v\rangle_{L^2}=\int_{\sigma(-\D)}\frac1{\lam+1} \left\langle u, dE(\lam )f\right\rangle_{L^2}$$ when $u,f  $ are continuous functions with compact support. The theorem then follows by 
  the general theory of self-adjoint operators 
  \cite[Section VIII.7]{RS}. 
\end{proof}


\subsection{Infinite  \Sig s.}\label{sec-iSig}

As a collection of first examples we consider the infinite \Sig s, where the spectrum was analyzed in \cite{Be,T98,Quint}. \begin{figure}[htb]\centering
\def\q{\includegraphics[height=26pt,width=30pt]{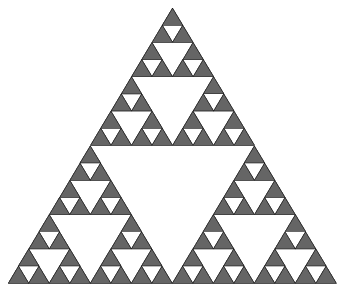}}
\def\qq{\put(0,0){\q}\put(28.6,0){\q}\put(14.3,25.2){\q}}
\reflectbox{{\begin{picture}(120, 177)(60, 0)
\put(0,0){\qq}\put(57.2,0){\qq}\put( 28.6,50.4){\qq}
\put(114.4,0){\qq}\put( 57.2,100.8){\qq}
\multiput(171.6,0)(-28.6,50.4){4}{\q}
\end{picture}}}\hskip10.08em\begin{picture}(120, 177)(60, 0)
\put(0,0){\qq}\put(57.2,0){\qq}\put( 28.6,50.4){\qq}
\put(114.4,0){\qq}\put( 57.2,100.8){\qq}
\multiput(171.6,0)(-28.6,50.4){4}{\q}
\end{picture}
\caption{A part of an infinite  \Sig.}\label{fig-iSig}
\end{figure}
\begin{figure}[htb]\centering
\def\q{\thicklines\qbezier(-2,-1)(0,0)(2,1)\qbezier(-2,1)(0,0)(2,-1)}
\def\qq{\put(0,0){\q}\put(0,2){\q}\put(0,4){\q}\put(0,6){\q}}
\def\qqq{\put(0,2){\qq}\put(0,15){\q}\put(0,20){\qq}}
\begin{picture}(400, 110)(0, 0)
\put(00,90){\q}
\put(00,45){\q}
\qqq\put(0,52){\qqq}
\def\q{\setlength{\unitlength}{1pt}\thicklines\qbezier(-1,-2)(0,0)(1,2)\qbezier(-1,2)(0,0)(1,-2)}
\put(-7,0){\SMALL0}\put(-7,45){\SMALL3}\put(-7,75){\SMALL5}\put(-7,90){\SMALL6}
\put(0,0){\vector(1,0){400}}\put(0,0){\vector(0,1){100}}
\multiput(0,90)(2.5,0){160}{\line(1,0){1}}
\multiput(0,45)(2.5,0){160}{\line(1,0){1}}
\multiput(21,0)(0,2.5){18}{\line(0,1){1}}\put(021,0){\q}
\multiput(129,0)(0,2.5){18}{\line(0,1){1}}\put(129,0){\q}
\multiput(225,0)(0,2.5){18}{\line(0,1){1}}\put(225,0){\q}
\multiput(355,0)(0,2.5){18}{\line(0,1){1}}\put(355,0){\q}
\multiput(60,0)(0,2.5){36}{\line(0,1){1}}
\multiput(90,0)(0,2.5){36}{\line(0,1){1}}\put(90,0){\q}
\multiput(272,0)(0,2.5){36}{\line(0,1){1}}
\multiput(308,0)(0,2.5){36}{\line(0,1){1}}\put(308,0){\q}
\put(0,0){\setlength{\unitlength}{15pt}\qbezier(0,0)(5,12.5)(10,0)}
\put(200,0){\setlength{\unitlength}{15pt}\qbezier(0,0)(6,12.5)(12,0)}
\def\qq{\put(0,0){\q}\put(2,0){\q}\put(4,0){\q}\put(6,0){\q}}
\def\qqq{\put(2,0){\qq}\put(15,0){\q}\put(20,0){\qq}}
{\setlength{\unitlength}{.5pt}\qqq\setlength{\unitlength}{.66666pt}\put(00042,0){\qqq}}
\put(150,0){
\def\qq{\put(-0,0){\q}\put(-2,0){\q}\put(-4,0){\q}\put(-6,0){\q}}
\def\qqq{\put(-2,0){\qq}\put(-15,0){\q}\put(-20,0){\qq}}
{\setlength{\unitlength}{.5pt}\qqq\setlength{\unitlength}{.66666pt}\put(-00042,0){\qqq}}}
\put(200,0){\setlength{\unitlength}{1.2pt}
\def\qq{\put(0,0){\q}\put(2,0){\q}\put(4,0){\q}\put(6,0){\q}}
\def\qqq{\put(2,0){\qq}\put(15,0){\q}\put(20,0){\qq}}
{\setlength{\unitlength}{.6pt}\qqq\setlength{\unitlength}{.8pt}\put(00042,0){\qqq}}
\put(150,0){
\def\qq{\put(-0,0){\q}\put(-2,0){\q}\put(-4,0){\q}\put(-6,0){\q}}
\def\qqq{\put(-2,0){\qq}\put(-15,0){\q}\put(-20,0){\qq}}
{\setlength{\unitlength}{.6pt}\qqq\setlength{\unitlength}{.8pt}\put(-00042,0){\qqq}}}
}
\end{picture}
\caption{An illustration to the computation of the spectrum on the  infinite \Sig. The curved lines show the graph of the function $\fR(\cdot)$, the vertical axis contains the spectrum of $\sigma(-\D_\Go)$ and the horizontal axis contains the spectrum  $\sigma(-\D)$.}\label{fig-iSig-s}
\end{figure}
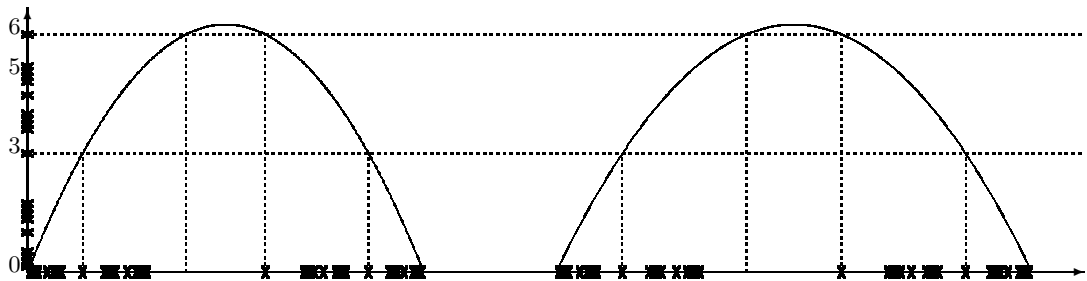
First, note that up to a natural isometry there is one \iSig\ with a distinguished boundary point (and hence it is not a fractafold), and there are uncountably many non-isometric \iSig s which are fractafolds (see \cite{T98} for more detail). 

If an \iSig\ fractafold is build in a self-similar way, as described in \cite{St98,T98}, then the spectrum on $\Go$ is pure point with two infinite series of eigenvalues of infinite multiplicity. One series of eigenvalues consists of isolated points which  accumulate to the Julia set $\mathcal J_R$ of the polynomial $R$, and the other series of points are located on the edges of the gaps of this Julia set (the Julia set in this case is a real Cantor set of one dimensional Lebesgue measure zero). The set of eigenvalues $\Sigma_0$ on $\Go$ consists of $6$ and all the preimages of $5$ and $3$ under the inverse iterations of $R$. In this case formula \eqref{e8} is the same as the formulas for eigenprojections in \cite{T98}. The illustration to the computation of the spectrum in Theorem~\ref{thm-main} is shown in Figure~\ref{fig-iSig-s}, 
where the graph of the function $\fR$ is shown schematically and the location of eigenvalues is denoted by small crosses. The spectrum $\sigma(-\D)$ is shown on the horizontal axis and the set of eigenvalues $\Sigma_0$ of $-\D_\Go$ is shown on the vertical axis.

A different \iSig\ fractafold can be constructed using two copies of  an \iSig\ with a boundary point, and joining these copies at the boundary. This fractal first was considered in \cite{BP}, and has a natural axis of symmetry between left and right copies. Therefore we can consider symmetric and anti-symmetric functions with respect to these symmetries. It was proved in 
\cite{T98} that the spectrum of the \Lp\ restricted to the symmetric part is pure point with a complete set of eigenfunctions with compact support. For the anti-symmetric part the compactly supported eigenfunctions are not complete, and it was proved in \cite{Quint} that the \Lp\ on $\Go$ 
has a singularly continuous component in the spectrum, supported on $\sJ_R$, of spectral multiplicity one. As a corollary of these and our  results we have the following proposition.

\BB{proposition}{prop-BP}{On the Barlow-Perkins  \iSif\ the spectrum of the \Lp\ consists of a dense set of eigenvalues $\fR^{-1}(\Sigma_0)$ of infinite multiplicity and of a singularly continuous component of spectral multiplicity one  supported on $\fR^{-1}(\sJ_R)$. }

%
%


\section{General infinite fractafolds and the main results}\label{sec-general}


Consider a fractafold with cell graph $\Gamma$, so $\G$ is an arbitrary infinite 3-regular graph.  The spectrum of $-\Delta_\Gamma$ is contained in [0,6], and by the spectral theorem there exist projection operators $E_I$ corresponding to intervals $I \subseteq [0,6]$.  Because we are in a discrete setting we can say a lot more.  There is a kernel function $E_I$ on $\G\times\G$ such that
\begin{equation}\label{4.1}
E_I f(a)=\displaystyle\sum_{b\in\G}E_I(a,b)f(b) \end{equation} 
and $I\to E_I(a,b)$ is a signed measure for each fixed $a, b$.  Since there are a countable number of such measures, we can find a single positive measure $\mu$ on [0,6] such that
\begin{equation}\label{4.2}
E_I(a,b)=\int_I P_\lambda(a,b)d\mu (\lambda) \end{equation}
for a function $P_\lambda(a,b)$ defined almost anywhere with respect to $\mu$ (so $P_\lambda(a,b)$ is just the Radon-Nykodim derivative of $E_I(a,b)$ with respect to $\mu$).  In fact, by a theorem of Besicovitch
\begin{equation}\label{4.3}
P_\lambda(a,b)=\lim_{\epsilon\to0}\frac{E_{[\lambda-\epsilon,\lambda+\epsilon]}(a,b)}{\mu([\lambda-\epsilon,\lambda+\epsilon])} \end{equation}
for $\mu-a.e. \lambda$.
(if $\mu$ is absolutely continuous this is just the Lebesgue differential of the integral theorem).
It follows from \eqref{4.3} that
\begin{equation}\label{4.4}
-\D_\G P_\lambda (\cdot,b)=\lambda P_\lambda(\cdot,b) \end{equation}
for $\mu-a.e. \lambda$. 
Thus if we define the pointwise projections
\begin{equation}\label{4.5}
P_\lambda f(a) = \displaystyle\sum_{b\in \G}P_\lambda (a,b)f (b) \end{equation}
then the spectral resolution is
\begin{equation}\label{4.6}
f=\int_\Sigma P_\lambda f d\mu(\lambda) \end{equation}
with
\begin{equation} \label{4.7}
-\D_\G P_\lambda f = \lambda P_\lambda f, \end{equation}
where $\Sigma \subseteq [0,6]$ is  the spectrum.  In other words, \eqref{4.6} represents a general function $f$ (we may take $f \in \ell^2(\G)$, or more restrictively a function of finite support) as an integral of $\lambda$-eigenfunctions.  Note that typically $P_\lambda f$ is not in $\ell^2(\G)$.  Also, the measure $\mu$ and the kernel $P_\lambda$ are not unique since one may be multiplied by $g(\lambda)$ and the other by $\frac{1}{g(\lambda)}$ for any positive function.  We are not aware of any way to make a ``canonical" choice to eliminate this ambiguity.

We also observe that the measure $\mu$ does not have a discrete atom at $\lambda=6$.  In other words, there are no $\ell^2(\G)$ 6-eigenfunctions. 
Indeed, for a 3-regular graph, there exist 6-eigenfunctions if an only if the graph is bipartite, in which case the 6-eigenfunction alternates $\pm1$ on the two parts.  Since we are assuming $\G$ is infinite, this eigenfunction is not in $\ell^2(\G)$.

Let $\Go$ denote the edge graph of $\G$.  Then $\Go$ is 4-regular.  Let $\D_\Go$ denote its Laplacian.  Define \begin{equation}\label{4.8}\tilde{P_\lambda}(x,y)=\frac{1}{6-\lambda}\displaystyle\sum_{a\in x} \displaystyle\sum_{b\in y}P_\lambda(a,b)\end{equation} (there are 4 terms in the sum).
Let $E_6$ denote the space of 6-eigenfunctions in $\ell^2(\Go)$ (this may be 0) and write $\tilde{P_6}$ for the orthogonal projection of $\ell^2(\Go)$ onto $E_6$.
\begin{theorem}\label{thm4.1}
The spectral resolution of $-\Delta_\Go$ is given by 
\begin{equation}\label{4.9}F=\tilde{P_6}F+\int_\Sigma \tilde{P_\lambda}Fd\mu(\lambda)\end{equation}
where
\begin{equation}\label{4.10} -\D_\Go \tilde{P_\lambda} F = \lambda \tilde{P_\lambda}F\end{equation}
for $\mu-a.e.\lambda$, and
\begin{equation}\label{4.11} \tilde{P_\lambda}F(x)=\displaystyle\sum_{y\in\Go}\tilde{P_\lambda}(x,y)F(y).\end{equation}
In particular, $spect(-\D_\Go)=\Sigma$ (if $E_6=0$) or $\Sigma \cup \{6\}$.
\end{theorem}
For the proof we require some lemmas.

Following \cite{St10} we define the sum operators $$S_1:\ell^2(\G)\to\ell^2(\G_0)$$ and 
$$S_2:\ell^2(\G_0)\to\ell^2(\G)$$ by 
\BEQ{e-4-1}{S_1f(x)=f(a)+f(b)\qquad\text{if $x$ is the edge $(a,b)$}} 
and 
\BEQ{e-4-2}{S_2F(a)=F(x)+F(y)+F(z)\qquad\text{if $x,y,z$ are the edges containing $a$}.} 

\BB{lemma}{lem-4-1}{$S_2S_1=6I+\D_\G$ and $S_1S_2=6I+\D_{\G_0}$. In particular, 
$S_2S_1$ is invertible, $S_1$ is one-to-one and $S_2$ is onto.}
\begin{proof}The formulas for $S_2S_1$ and $S_1S_2$ are simple computations. Since there are no 6-eigenfunctions in $\ell^2(\G)$, we obtain the invertability of $S_2S_1$  (see also \cite{St10}).\end{proof} 

It follows from Lemma~\ref{lem-4-1} that $E_6=(
\text{Im} S_1)^\perp$ and $\ell^2(\G_0)=\text{Im} S_1\oplus E_6$. 

\BB{lemma}{lem-4-2}{For any $\lam\neq6$, $-\D_\G f=\lam f$ 
if and only if $-\D_{\G_0}S_1f=\lam S_1f$. In particular, $sp(-\D_{\G_0})=sp(-\D_{\G})\cup\{6\}$.}

\begin{proof}
Suppose $-\D_\G f=\lam f$. Since $-\D_\G = 6I-S_2S_1$ we have $-S_2S_1f=(\lam-6)f$. Apply $S_1$ to this identity and use $-\D_{\G_0}=6I-S_1S_2$ to obtain 
$-\D_{\G_0}S_1f=\lam S_1f$. Similarly, we can reverse the implications. Note that the condition $\lam\neq6$ implies that $S_1f$ is not identically zero 
(see also \cite{St10}).\end{proof}

\begin{lemma}\label{lem4.4}
Let $F\in \ell^2(\Go)$ be orthogonal to $E_6$ (if $E_6$ is nontrivial).  Then $F=S_1f$ for \begin{equation}\label{4.14} f=(6I+\D_\G)^{-1}S_2F_1.\end{equation}
Moreover we have \begin{equation}\label{4.15} \tilde{P_\lambda}=\frac{1}{6-\lambda}S_1P_\lambda S_2. \end{equation}\end{lemma}
\begin{proof}
For $f$ defined by \eqref{4.4} we have $S_2S_1f=S_2F$ by Lemma \ref{lem-4-1}.  Since $S_2$ is injective on $E^\perp_6$ and $S_1f\in E^\perp_6$ we conclude $S_1f=F$.  

By definition, $\tilde{P_\lambda}F(x)=\displaystyle\sum_{y\in\Go}\frac{1}{6-\lambda} \displaystyle\sum_{a\in x}\displaystyle\sum_{b\in y} P_\lambda(a,b)F(y)$ and this is equivalent to \eqref{4.15} by the definition of $S_1$ and $S_2$.
\end{proof}

\begin{proof}[Proof of Theorem \ref{thm4.1}]
It suffices to establish \eqref{4.9} for $F\in E^\perp_6$.  For $f$ defined by \eqref{4.14}, we apply $S_1$ to \eqref{4.6} to obtain $$F=\int S_1 P_\lambda f d\mu(\lambda)=\int S_1P_\lambda(6I+\D_G)^{-1}S_2Fd\mu=\int\frac{1}{6-\lambda}S_1P_\lambda S_2 F d\mu(\lambda)$$
since $P_\lambda(6I+\D_\G)^{-1}=\frac{1}{6-\lambda}P_\lambda$.  Then \eqref{4.9} follows by \eqref{4.15}.  We obtain \eqref{4.10} from \eqref{4.7} and Lemma \ref{lem-4-2}.
\end{proof}

In order to give an explicit form of the spectral resolution for any particular $\G$, we need to solve two problems:
\begin{itemize}
\item[(a)] Find an explicit formula for $P_\lambda(a,b)$;
\item[(b)] Give an explicit description of $E_6$ and the projection operator $\tilde{P_6}$.
\end{itemize}

In addition, there is one more problem we would like to solve in order to obtain an explicit Plancherel formula.  We can always write
\begin{equation}\label{4.16}
||f||^2_{\ell^2(\G)}=\int_\Sigma <P_\lambda f,f>d\mu(\lambda)
\end{equation}
and
\begin{equation}\label{4.17}
||F||^2_{\ell^2(\Go)}=||\tilde{P_6}F||^2_2+\int_\Sigma <\tilde{P_\lambda}F,F>d\mu(\lambda)
\end{equation}
for a reasonable dense space of functions $f$ and $F$ (certainly finitely supported functions will do).  What we would like is to replace $<P_\lambda f,f>$ and $<\tilde{P_\lambda}F,F>$ by expressions only involving $P_\lambda f$ and $\tilde{P_\lambda}F$ and some inner product on a space of $\lambda$-eigenfunctions.  Note that from \eqref{4.2} and the fact that $E_I$ is a projection operator we have 
\begin{equation}\label{4.18}
<P_\lambda f,f>=\lim_{\epsilon \to 0} \mu([\lambda-\epsilon, \lambda+\epsilon])||\frac{1}{\mu([\lambda-\epsilon, \lambda+\epsilon])}E_{[\lambda-\epsilon, \lambda+\epsilon]}f||^2_2
\end{equation}
for $\mu-a.e.\lambda$.  This suggests the following conjecture,
\begin{conjecture}\label{conj4.5}
For $\mu-a.e. \lambda$ there exists a Hilbert space of $\lambda$-eigenfunctions $\xi_\lambda$ with inner product $<,>_\lambda$ such that $P_\lambda f \in \xi_\lambda$ for $\mu-a.e.\lambda$ for every $f\in \ell^2(\G)$, and
\begin{equation}\label{4.19}
<P_\lambda f,f>=<P_\lambda f, P_\lambda f>_\lambda. \end{equation}
Moreover a similar statement holds for $<\tilde{P_\lambda} F,F>.$
\end{conjecture}

Our last problem is then
\begin{itemize}
\item[(c)] Find an explicit description of $\xi_\lambda$ and its inner product, and transfer this to $\tilde{\xi_\lambda}$ of $\Go$.
\end{itemize}

\section{The Tree Fractafold}\label{sec5}

In this section we study in detail the spectrum of the \Lp\ on the tree fractafold TSG 
(Figure~\ref{fig-sig-tree}
)\begin{figure}[hbt]\centering
\includegraphics{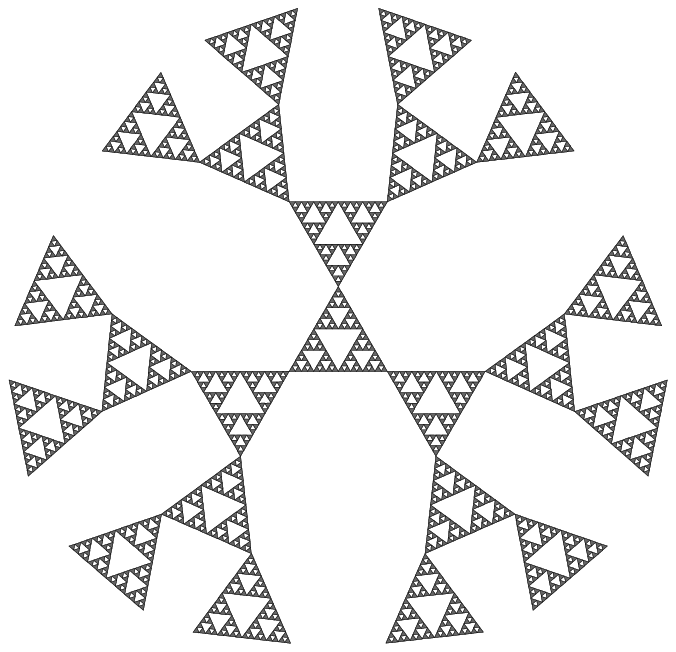}
\caption{A part of the infinite \Sif\ based on the binary tree.}\label{fig-sig-tree}
\end{figure}
 whose cell graph $\G$ is the 3-regular tree. In a sense this example is the ``universal covering space'' of all the other examples, if we ``fill in'' all copies of SG with triangles.

We begin by solving problem (b).

\begin{lemma}\label{lem5.1}
For any fixed $z$ in $\Go$ define 
\begin{equation*} F_z(x)=\frac{1}{\sqrt{3}}(-\frac{1}{2})^{d(x,z)}.\end{equation*}
Then $F_z\in\ell^2(\Go)$ with $||F_z||_{\ell^2(\Go)}=1$ and $F_z\in E_6$.  \end{lemma}
\begin{proof} Note that $z$ has 4 neighbors $\{y_1, y_2, y_3, y_4\}$ in $\Go$ with $d(y_j,z)=1$, so
\begin{eqnarray*}
-\D_\Go F_z(z)&=&4F_z(z)-\displaystyle\sum_{j=1}^4F(y_j)\\
&=&\frac{1}{\sqrt{3}}(4(-\frac{1}{2})^0-4(-\frac{1}{2})^1)=\frac{6}{\sqrt{3}}=6F_z(z)
\end{eqnarray*}
verifying the 6-eigenvalue equation at $z$.

On the other hand, if $x\neq z$ then the 4 neighbors $\{y_1, y_2, y_3, y_4\}$ of $x$ may be permuted so that $d(y_1,z)=d(x,z)-1$, $d(y_2,z)=d(x,z)$, and $d(y_3,z)=d(y_4,z)=d(y,z)+1$.  It follows that
\begin{eqnarray*}
-\D_\Go F_z(x)=4F_z(x)-\displaystyle\sum_{j=1}^4F_z(y_j)=F_z(x)(4-(-2+1-2 \cdot\frac{1}{2}))=6F_z(x)
\end{eqnarray*}
verifying the 6-eigenvalue equation at $x$.  Finally
\begin{eqnarray*}
||F_z||^2_{\ell^2(\Go)}=\frac{1}{3}(1+4\cdot(\frac{1}{2})^2+8\cdot(\frac{1}{4})^2+\ldots)=\frac{1}{3}(1+1+\frac{1}{2}+\frac{1}{4}+\ldots)=1
\end{eqnarray*}
(See Figure \ref{fig5.1}).
\end{proof}

\begin{figure}[htb]
\begin{center}
\begin{picture}(150,100)(-75,-55)\thicklines
\setlength{\unitlength}{6pt}
\put(-8,-8){\line(1,1){16}}
\put(-8,-8){\line(-1,1){4}}
\put(-12,-4){\line(1,0){24}}
\put(8,-8){\line(-1,1){16}}
\put(-12,4){\line(1,0){24}}
\put(-12,4){\line(1,1){4}}
\put(8,-8){\line(1,1){4}}
\put(12,4){\line(-1,1){4}}
\put(1.3,-.3){$1$}
\put(-6.5,2){$-\frac{1}{2}$}
\put(4,2){$-\frac{1}{2}$}
\put(-6.5,-2.5){$-\frac{1}{2}$}
\put(4,-3){$-\frac{1}{2}$}
\put(-14,4){$\frac{1}{4}$}
\put(12,4){$\frac{1}{4}$}
\put(-14,-4){$\frac{1}{4}$}
\put(12,-4){$\frac{1}{4}$}
\put(-8,9){$\frac{1}{4}$}
\put(8,9){$\frac{1}{4}$}
\put(-8,-10){$\frac{1}{4}$}
\put(8,-10){$\frac{1}{4}$}
\end{picture}
\end{center}
\caption{Values of $\sqrt{3}F_z$ (the center point is $z$).}\label{fig5.1}\end{figure}
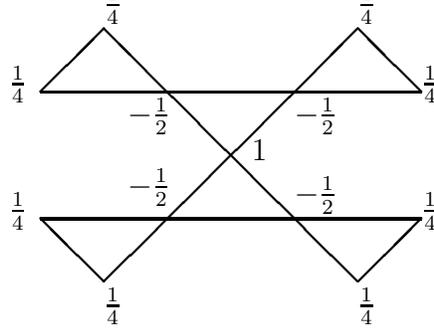

\begin{remark}
It is easy to see from the 6-eigenvalue equation that $F_z$ is the unique (up to a constant multiple) function in $E_6$ that is radial about $z$ (a function of $d(x,z)$).
\end{remark}
\begin{lemma}\label{lem5.3}
$\displaystyle\sum_xF_z(x)F_y(x)=\sqrt{3}F_z(y)$.
\end{lemma}
\begin{proof} Fix $z$.  Then the left side is a 6-eigenfunction of $y$ and is radial about $z$, so it must be a constant multiple of $F_z(y)$.  To compute the constant set $y=z$, and the left side is $1$ while $F_z(z)=\frac{1}{\sqrt{3}}$. \end{proof}

\begin{definition}\label{defn5.4}
Let $\tilde{P_6}(x,y)=\frac{1}{\sqrt{3}}F_x(y)=\frac{1}{3}(-\frac{1}{2})^{d(x,y)}$ and define the operator 
\begin{equation}
\tilde{P_6}F(x)=\displaystyle\sum_y\tilde{P_6}(x,y)F(y). \end{equation} \end{definition}

\begin{theorem}\label{thm5.1}
$\tilde{P_6}$ is the orthogonal projection $\ell^2(\Go)\to E_6$. \end{theorem}
\begin{proof}
Lemma \ref{lem5.3} shows $\tilde{P_6}F_z=F_z$.  Now we claim that the functions $F_z$ span $E_6$.  Indeed, if $F$ is in $E_6$ and is orthogonal to $F_z$, then we can radialize $F$ about $z$ to obtain a function $\tilde{F}$ that is still in $E_6$ and orthogonal to $F_z$.  Since $\tilde{F}$ must be a multiple of $F_z$ it follows that it is identically zero.  Since $\tilde{F}(z)=F(z)$ it follows that $F(z)=0$.  Since this holds for every $z$, we have shown that the orthogonal complement of the span of $F_z$ is zero.  This shows $\tilde{P_6}$ is the identity on $E_6$.  Also $\tilde{P_6}E^\perp_6=0$ by the orthogonality of different parts of the spectrum.
\end{proof}

Note that $\{F_z\}$ is not an orthonormal basis of $E_6$, since $<F_z,F_y>=\sqrt{3}F_z(y)$ by Lemma \ref{lem5.3}.
The next result shows that it is a tight frame.
\begin{theorem}\label{thm5.2}
For any $F\in E_6$ we have 
\begin{equation}
\displaystyle\sum_z|<F,F_z>|^2=3||F||^2_{\ell^2(\Go)}  \end{equation}
\end{theorem}
\begin{proof}
We may write $F=\displaystyle\sum_ya(y)F_y$.  Then $||F||^2_{\ell^2(\Go)}=\displaystyle\sum_y\displaystyle\sum_z a(y)\bar{a(z)}\sqrt{3}F_z(y)$.  But $<F,F_z>=\displaystyle\sum_y a(y)\sqrt{3}F_z(y)$ and so
\begin{eqnarray}
\displaystyle\sum_z|<F,F_z>|^2&=&3\displaystyle\sum_z\displaystyle\sum_y\displaystyle\sum_{y'} a(y)\overline{a(y')}F_y(z)F_{y'}(z) \notag \\
&=&3\displaystyle\sum_y\displaystyle\sum_{y'}a(y)\overline{a(y')}F_{y'}(y) =3||F||^2_{\ell^2(\Go)}\label{5.2}
\end{eqnarray}
\end{proof}

It follows from polarizing \eqref{5.2} that we may also write $\tilde{P_6}F=\frac{1}{3}\Sum_z<F,F_z>F_z$.

The solution of problem $(a)$ is due to Cartier \cite{Cartier}.  
We outline the solution   following \cite{F-TN}.

\begin{definition}\label{defn5.5}
Let $z\in \mathbb{C}$ with $2^{2z-1}\neq1$.  Let $c(z)=\frac{1}{3}\frac{2^{1-z}-2^{z-1}}{2^{-z}-2^{z-1}}$, $c(1-z)=\frac{1}{3}\frac{2^{-z}-2^z}{2^{-z}-2^{z-1}}$ and $\varphi_z(n)=c(z)2^{-nz}+c(1-z)2^{-n(1-z)}$.
\end{definition}
\begin{remark}
Note that $c(z)$ and $c(1-z)$ are characterized by the identities $c(z)+c(1-z)=1$ and $c(z)2^{-z}+c(1-z)2^{z-1}=c(z)2^z+c(1-z)2^{1-z}$ which imply $\varphi_z(0)=1$ and $\varphi_z(1)=\varphi_z(-1)$.
\end{remark}
\begin{theorem}\label{thm5.3}
For any fixed $y\in\G$, let $f_y(x)=\varphi_z(d(x,y))$.  Then
\begin{equation}
-\D_\G f_y=(3-2^z-2^{1-z})f_y \end{equation}
and $f_y$ may be characterized as the unique $(3-2^z-2^{1-z})$-eigenfunction that is radial about $y$ and satisfying $f_y(y)=1$.  \end{theorem}
\begin{proof}
Uniqueness follows from the eigenvalue equation.  To verify the eigenvalue equation we do the computation separately for $x\neq y$ and $x=y$.  For $x\neq y$ note that $x$ has two neighbors, $x_1$ and $x_2$, with $d(x_1,y)=d(x_2,y)=d(x,y)+1$ and one neighbor, $x_3$, with $d(x_3,y)=d(x,y)-1$ so the eigenvalue equation is immediate.  On the other hand $y$ has three neighbors, $x_1,x_2,x_3$, with $d(x_j,y)=1$, and the eigenvalue equation follows from $\varphi_z(1)=\varphi_z(-1)$.  \end{proof}

Note that there is no choice of $z$ that
will make $f_y$ belong to $\ell^2(\G)$.  However, the choice $z=\frac{1}{2}+i t$ gets close.  Indeed $|\varphi_{\frac{1}{2}+it}(d(x,y))|^2 \approx \displaystyle\sum_n 2^n \cdot 2^{-n}$ just diverges.  So it is natural to conjecture that these eigenfunctions give the spectral resolution of $-\D_\G$ on $\ell^2(\G)$.  In fact the following proposition is the content of Theorem 6.4 on p. 61 of \cite{F-TN}.

\begin{proposition}\label{prop5.7}
By periodicity we may restrict $t$ to $0\leq t \leq \frac{\pi}{\log 2}$.  Write $\lambda(t)=3-2\sqrt{2} \cos(t \log 2)=3-2^{\frac{1}{2}+it}-2^{\frac{1}{2}-it}$ and  $\sum=[3-2\sqrt{2},3+2\sqrt{2}]\approx[0.17,5.83]\subsetneq[0,6]$.  Define
\begin{equation} P_tf(x)=\displaystyle\sum_y \varphi_{\frac{1}{2}+it}(d(x,y))f(y). \end{equation}
Note that $-\D_\G P_t f=\lambda(t)P_t(f)$.  Then
\begin{equation} f(x)=\int_0^{\frac{\pi}{\log 2}}P_tf(x)dm(t) \end{equation}
for the measure
\begin{equation}
dm(t)=\frac{\log 2}{3 \pi} \left|c(\frac{1}{2}+it)\right|^{-2}dt=\frac{(3 \log 2)\sin^2(t \log 2)}{\pi(1+2 \sin^2(t \log 2))} dt. \end{equation}
\end{proposition}

It is convenient to change notation so that the eigenvalue $\lambda$ rather than $t$ is the parameter.  We easily compute $t=\frac{1}{\log 2} \cos^{-1}\left(\frac{3-\lambda}{2\sqrt{2}}\right)$.

Note that $$d\lambda=2\sqrt{2}\log2\sin(t \log2)dt, \qquad \sin^2(t\log 2)=({-\lambda^2+6\lambda-1})/{8},$$  and $$ 1+2\sin^2(t\log2)=({-\lambda^2+6\lambda+3})/{4}.$$
If we write $P_\lambda=P_t$ then the spectral resolution is $$f(x)=\int_{3-2\sqrt2}^{3+2\sqrt2}P_\lambda f(x)dm(\lambda)$$ for $$dm(\lambda)=\frac{3\sqrt{-\lambda^2+6\lambda-1}}{\sqrt{2}\pi(-\lambda^2+6\lambda+3)}d\lambda.$$

Now suppose $F\in\ell^2(\Go)$ lies in $E^\perp_6$.  Then we may write $F=S_1f$ for $f=(6I+\D_\G)^{-1}S_2F$ in $\ell^2(\G)$.
Indeed we know that 6 is in the resolvent of $-\D_\G$ so $f$ is well-defined, and then $S_2S_1f=S_2F$ by Lemma~\ref{lem-4-1}.  
Since $S_2$ is injective on $E_6^{\perp}$ and $S_1f\in E_6^\perp$ we conclude $S_1f=F$.

By Proposition \ref{prop5.7} we have
\begin{equation}  S_1f=\int_\Sigma S_1P_\lambda f dm(\lambda),\end{equation}
and of course $-\D_\Go S_1P_\lambda f=\lambda S_1P_\lambda f$ by Lemma~\ref{lem-4-2}, 
so we define $\tilde{P_\lambda}F=S_1P_\lambda f$ and we obtain the spectral resolution of $F$:
\begin{equation} F=\int_\Sigma\tilde{P_\lambda}F dm(\lambda). \end{equation}
Note that $P_\lambda(6I+\D_\G)^{-1}=\frac{1}{6-\lambda}P_\lambda$ so $\tilde{P_\lambda}F=\frac{1}{6-\lambda}S_1P_\lambda S_2F$.

We may write this quite explicitly as follows:
\begin{lemma}\label{lem5.8}
Define
\begin{equation}
\psi_z(n)=\tilde{c}(z)2^{-nz}+\tilde{c}(1-z)2^{-n(1-z)} \end{equation}
for $\tilde{c}(z)=(2+2^{-z}+2^z)c(z)$.  Note that 
 $\psi_z(n)=2\varphi_z(n)+\varphi_z(n+1)+\varphi_z(n-1)$. 
 Then 
\begin{equation}
S_1P_\lambda S_2 F(x)=\frac{1}{3} \displaystyle\sum_y \psi_{\frac{1}{2}+it}(d(x,y))F(y). \end{equation}
\end{lemma}
\begin{proof}
$S_2F(b)=\displaystyle\sum_{y\sim b}F(y)$.  There are 3 terms in the sum, and $y\sim b$ means the edge $y$ joins $b$ and one of its neighbors in $\G$.  Then we compute
\begin{equation}
P_\lambda S_2 F(a)=\Sum_{b\in \G}\Sum_{y\sim b} \varphi_{\frac{1}{2}+it}(d(a,b))F(y)
\end{equation}
and 
\begin{equation} 
S_1P_\lambda S_2 F(x)=\Sum_{a\sim x}\Sum_{b\in\G}\Sum_{y\sim b} \varphi_{\frac{1}{2}+it}(d(a,b))F(y)
\end{equation}
where $a\sim x$ means that $a$ is one of the vertices in the edge $x$.  Suppose $x\neq y$ and let $n=d(x,y)$ with $n\geq1$, (Figure \ref{fig-5-3} shows the $\Go$ graph for $n=2$).

\begin{figure}[htb]
\begin{center}
\begin{picture}(100,105)(-50,-50)\thicklines
\setlength{\unitlength}{7pt}
\put(-16,0){\line(1,0){32}}
\put(-16,0){\line(1,-1){4}}
\put(-12,-4){\line(1,1){12}}
\put(0,8){\line(-1,0){8}}
\put(-8,8){\line(1,-1){16}}
\put(0,-8){\line(1,0){8}}
\put(0,-8){\line(1,1){12}}
\put(12,4){\line(1,-1){4}}
\put(-13,-2){$a_2$}
\put(-8,-1.2){$x$}
\put(-5,1.5){$a_1$}
\put(3,-2){$b_1$}
\put(11,1.5){$b_2$}
\put(8.5,-1.2){$y$}
\end{picture}
\end{center}
\caption{Graph $\Go$}\label{fig-5-3}\end{figure}

Then $x\sim a_1$ and $x\sim a_2$ while $y\sim b_1$ and $y\sim b_2$ with $d(a_1,b_2)=d(a_2,b_1)=n$, $d(a_1, b_1)=n-1$, and $d(a_2,b_2)=n+1$.  The result follows in this case.  When $x=y$ we have $d(x,y)=0$ and $a_1=b_1$, $a_2=b_2$ so $d(a_1,b_2)=d(a_2,b_1)=1$ and $d(a_1,b_1)=d(a_2,b_2)=0$.  The result follows because $\varphi_{\frac{1}{2}+it}(-1)=\varphi_{\frac{1}{2}+it}(1)$.
\end{proof}
\begin{theorem}\label{thm5.4}
For any $F\in\ell^2(\Go)$ we have the explicit spectral resolution
\begin{equation}
F=\tilde{P_6}F+\int_\Sigma\tilde{P_\lambda}Fdm(\lambda) \end{equation}
for
\begin{equation}
\tilde{P_\lambda}F(x)=\frac{1}{3(6-\lambda)}\Sum_y \psi_{\frac{1}{2}+it}(d(x,y))F(y) .\end{equation}
\end{theorem}
The Theorem follows by combining Lemma \ref{lem5.8} and Proposition \ref{prop5.7}.  We note that the proof of Proposition \ref{prop5.7} involves an explicit computation of the resolvent $(\lambda I+\D_\G)^{-1}$ for $\lambda$ outside the spectrum of $-\D_\G$, followed by a contour integral to obtain the spectral resolution from the resolvent.  We sketch some of these ideas and then show how to carry out a similar proof of Theorem \ref{thm5.4}.

On $\ell^2(\G)$ we define
\begin{equation} H_zf(a)=\Sum_b2^{-zd(a,b)}f(b).\end{equation}
A direct computation shows
\begin{equation} (\lambda I+\D_\G)H_zf=(2^{-z}-2^z)f \end{equation}
for $\lambda=3-2^z-2\cdot2^{-z}$.

Note that $\frac{3-\lambda}{2\sqrt{2}}=\cosh((z-\frac{1}{2})\log2)$, and in order to have $H_z$ bounded on $\ell^2(\G)$ we need $\Re z>\frac{1}{2}$.  This shows $spect(-\D_\G)=\Sigma$ and $(\lambda I +\D_\G)^{-1}=\frac{1}{2^{-z}-2^z}H_z$ for $z \notin\Sigma$.

On $\ell^2(\Go)$ we define
\begin{equation}
\tilde{H_z}F(x)=\Sum_y2^{-zd(x,y)}F(y). \end{equation}

\begin{lemma}\label{lem5.9}
$spect(-\D_\Go)^{-1}=\Sigma \cup \{6\}$ and $(\lambda I+\D)^{-1}=\frac{1}{2\cdot2^{-z}-2^z-1}\tilde{H_z}$ for $z\notin spect(-\D_\Go)$.
\end{lemma}
\begin{proof}
Note that $\tilde{H_z}$ is bounded on $\ell^2(\Go)$ for $\Re z>\frac{1}{2}$.  Also $\lambda=6$ corresponds to $z=1+\frac{\pi i}{\log 2}$ for which $2\cdot2^{-z}-2^z-1=2(-\frac{1}{2})-(2)-1=0$.  Now fix $x$ and consider its four neighbors, $x_1,x_2,x_3,x_4$ (so $d(x,x_j)=1$).  For any fixed $y\neq x$ we may order them so that $d(x_1,y)=d(x_2,y)=d(x,y)+1$, $d(x_3,y)=d(x,y)$, $d(x_4,y)=d(x,y)-1$.  It follows that 
\begin{eqnarray}
&(\lambda I+\D_\Go)\tilde{H_z}F(x)=(\lambda-4)\tilde{H_z}F(x)+\Sum_j\tilde{H_z}F(x_j) \notag \\&
=(\lambda-4)F(x)+\Sum_j2^{-z}F(x) \notag +(\lambda-4)\Sum_{y\neq x}2^{-zd(x,y)}F(y)+\Sum_j \Sum_{y\neq x}2^{-zd(x,y)}F(y)\notag \\&
=(2\cdot2^{-z}-2^z-1)F(x)
\end{eqnarray}
and the result follows.
\end{proof}

For $f\in \ell^2(\G)$, we have
\begin{equation}\label{e-gamma} f=\frac{1}{2\pi i} \int_\gamma (\lambda I+\D_\G)^{-1} f d\lambda \end{equation}
for any contour $\gamma$ that circles $\Sigma$ once in the counterclockwise direction.  We choose $\gamma$ as shown and take the limit as $\delta \to 0^+$.  The contribution from the vertical 
segments goes to zero so
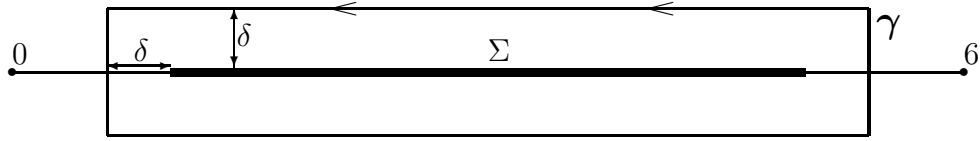
\begin{figure}[htb]\centering
\begin{picture}(360, 80)(0, -40)
\thicklines
\setlength{\unitlength}{60pt}
\put(0,0){\line(1,0){6}}
\put(.6,-.4){\line(1,0){4.8}\line(0,1){0.8}}
\put(5.4,.4){\line(-1,0){4.8}\line(0,-1){0.8}}
\linethickness{3pt}
\put(1,0){\line(1,0){4}}
\thinlines
\put(.6,.04){\vector(1,0){.4}}\put(1,.04){\vector(-1,0){.4}}
\put(1.4,.03){\vector(0,1){.37}}\put(1.4,.4){\vector(0,-1){.37}}
\multiput(0,0)(6,0)2{\circle*{.05}}
\put(4,.4){\setlength{\unitlength}{1pt}\line(4,1){10}}
\put(4,.4){\setlength{\unitlength}{1pt}\line(4,-1){10}}
\put(2,.4){\setlength{\unitlength}{1pt}\line(4,1){10}}
\put(2,.4){\setlength{\unitlength}{1pt}\line(4,-1){10}}
\setlength{\unitlength}{1pt}
\put(0,3){$0$}
\put(360,3){$6$}
\put(180,4){$\Sigma$}
\put(327,15){\large\mathversion{bold}$\gamma$}
\put(46,4){$\delta$}
\put(85,10){$\delta$}
\end{picture}
\caption{The contour $\gamma$ for integration in \eqref{e-gamma}. }\label{fig-contour}
\end{figure}
\begin{equation}
f=\lim_{\delta \to 0^+} \frac{1}{2 \pi i} \int_\Sigma \left( (\lambda - i \delta +\D_\G)^{-1}-(\lambda+ i \delta +\D_\G)^{-1}\right)f d\lambda. \end{equation}
If $z=\frac{1}{2} + \epsilon + i t$ for $\epsilon >0$ then $\frac{3-\lambda}{2\sqrt{2}}=\cos(t \log2 - i \epsilon \log2)$ and
\begin{equation}
\lambda=3-2\sqrt{2} \cos(t \log2) \cosh(\epsilon \log2)-i2\sqrt{2} \sinh(\epsilon \log2)\sin(t \log2). \end{equation}
For $t>0$ we have $\lambda \approx 3-2\sqrt{2} \cos(t \log2)-i\delta$, while for $t<0$ we have $\lambda \approx 3-2\sqrt{2} \cos(t \log2)+i\delta$ with $\delta>0$.  Thus
\begin{equation}
\lim_{\delta \to 0^+}\left( (\lambda - i \delta +\D_\G)^{-1}-(\lambda + i \delta +\D_\G)^{-1}\right) f = \frac{1}{2^{-\frac{1}{2}- i t}-2^{\frac{1}{2}+i t}}H_{\frac{1}{2}+i t} f -\frac{1}{2^{-\frac{1}{2}+ i t}-2^{\frac{1}{2}-i t}}
H_{\frac{1}{2}-i t}f
\end{equation}
so we obtain
\begin{equation}
f=\frac{1}{2 \pi i} \int_0^{\frac{\pi}{\log2}} \left( \frac{1}{2^{-\frac{1}{2}-it}-2^{\frac{1}{2}+it}}H_{\frac{1}{2}+i t} f - \frac{1}{2^{-\frac{1}{2}+it}-2^{\frac{1}{2}-it}}H_{\frac{1}{2}-i t} f\right) 2\sqrt{2} \log2 \sin(t\log2)dt.
\end{equation}
This is the same as $f=\int_0^{\frac{\pi}{\log2}} P_t f dm(t)$.

For $F\in \ell^2(\Go)$ we have
\begin{equation}
F=\frac{1}{2\pi i} \int_\gamma (\lambda I+\D_\Go)^{-1}Fd\lambda + \frac{1}{2\pi i} \int_{\gamma'} (\lambda I+\D_\Go)^{-1}Fd\lambda
\end{equation}
where $\gamma$ is as before and $\gamma'$ is a small circle about 6.  Taking the limit we obtain
\begin{eqnarray}
F&=&\lim_{\delta\to0^+}\frac{1}{2\pi i} \int_\Sigma \left((\lambda-i \delta +\D_\Go)^{-1} F-(\lambda+i \delta +\D_\Go)^{-1} F\right)d\lambda\notag \\
&& +\lim_{\delta\to0^+}\frac{1}{2\pi i} \int_0^{2\pi}(6+\delta e^{i \theta} +\D_\Go)^{-1} F i \delta e^{i \theta} d\theta. 
\end{eqnarray}
As before we can write the first term as
\begin{equation}
\frac{\sqrt{2} \log2}{ \pi i} \int_\Sigma \left( \frac{1}{2^{\frac{1}{2}-it}-2^{\frac{1}{2}+it}-1}\tilde{H}_{\frac{1}{2}+i t} F - \frac{1}{2^{\frac{1}{2}+it}-2^{\frac{1}{2}-it}-1}\tilde{H}_{\frac{1}{2}-i t} F\right)  \sin(t\log2)dt.
\end{equation}
which we identify with $\int_\Sigma \tilde{P_\lambda} F dm(\lambda)$, while the second term is $\tilde{P_6}F$.


Next we discuss an explicit Plancherel formula on $\G$, given in terms of the modified mean inner product
\begin{equation}\label{5.28}
<f,g>_M=\displaystyle\lim_{N\to\infty} \frac{1}{N} \Sum_{d(x,x_0)\leq N}f(x)\overline{g(x)}. 
\end{equation}
We will deal with eigenspaces for which the limit exists and is independent of the point $x_0$.  Note that this is not the usual mean on $\G$, since the cardinality of the ball $\{x:d(x,x_0)\leq N\}$ is $O(2^n)$, but it is tailor made for functions of growth rate $O(2^{-d(x,x_0)/2})$, which is exactly the growth rate of our eigenfunctions.

We expect that analogous results are valid for $k$-regular trees 
for all $k$, but to keep the discussion simple we only deal with 
the case $k=3$ that we need for our applications. 

\begin{lemma}\label{lemV1}
For all $n$ and $t$
\begin{equation}\label{5.29}
\varphi_{\frac{1}{2}+it}(n)=\frac{1}{3}\left(3\cos(nt\log2)+\frac{\sin(nt\log2)}{\tan(t\log2)}\right)2^{-n/2} \end{equation}
\end{lemma}
\begin{proof}
From the definition,
\begin{equation*}
\varphi_{\frac{1}{2}+it}(n)=\left(2\Re(c(\frac{1}{2}+it)2^{-itn}\right)2^{-n/2}. \end{equation*}
The result follows from the explicit formula for $c(\frac{1}{2}+it)$ and some trigonometric identities. 
\end{proof}

In what follows we write $\varphi$ for $\varphi_{\frac{1}{2}+it}$ to simplify the notation.
\begin{lemma} \label{lemV2}
Let
\begin{equation}\label{5.30}
b(\lambda)=8+\frac{1}{\sin^2(t\log2)}=8\left(\frac{-\lambda^2+6\lambda}{-\lambda^2+6\lambda-1}\right).
\end{equation}
Then for any integers $k$ and $j$,
\begin{eqnarray}\label{5.31}
\displaystyle\lim_{N\to\infty}\frac{1}{N}\Sum_{n=1}^N 2^{n+\frac{k}{2}}\varphi(n)\varphi(n+k)&
=\displaystyle\lim_{N\to\infty}\frac{1}{N}\Sum_{n=1}^N 2^{n+j+\frac{k}{2}}\varphi(n+j+k)\varphi(n+j)\\
&=\frac{1}{18}b(\lambda)\cos(kt\log2). \notag
\end{eqnarray}
\end{lemma}
\begin{proof}
It is easy to see that \eqref{5.31} is independent of $j$, so we take $j=0$.  Then by \eqref{5.29} 
\begin{eqnarray}
2^{n+\frac{k}{2}}\varphi(n)\varphi(n{+}k)=\frac{1}{9}\left(3\cos(nt\log2){+}\frac{\sin(nt\log2)}{\tan(t\log2)}\right)\left(3\cos(nt\log2)\cos(kt\log2)\right.\hspace{1cm}\notag\\
\left.-3\sin(nt\log2)\sin(kt\log2)+\frac{\sin(nt\log2)\cos(kt\log2)}{\tan(t\log2)}+\frac{\cos(nt\log2)\sin(kt\log2)}{\tan(t\log2)}\right).\notag
\end{eqnarray}
Now use the following identities
\begin{equation*}
\displaystyle\lim_{N\to\infty}\frac{1}{N}\Sum_{n=1}^N\cos^2n\alpha=\displaystyle\lim_{N\to\infty}\frac{1}{N}\Sum_{n=1}^N\sin^2 n\alpha=\frac{1}{2} \end{equation*}
and
\begin{equation*}
\displaystyle\lim_{N\to\infty}\frac{1}{N}\Sum_{n=1}^N\cos n\alpha \sin n\alpha=0 \end{equation*}
to see that the limit in \eqref{5.31} equals
\begin{equation*}
\frac{1}{18}\left(9\cos(kt\log2)+\frac{3\sin(kt\log2)}{\tan(t\log2)}-\frac{3\sin(kt\log2)}{\tan(t\log2)}+\frac{\cos(kt\log2)}{\tan^2(t\log2)}\right)=\frac{1}{18}b(\lambda)\cos(kt\log2).
\end{equation*}
\end{proof}
\begin{lemma}\label{lemV3}
For any $\lambda$ in the interior of $\Sigma$ and $x_1\in\G$, $<P_\lambda \delta_{x_1},P_\lambda \delta_{x_1}>_M$ exists and is independent of the base point $x_0$, and
\begin{equation}\label{5.32}
<P_\lambda \delta_{x_1},P_\lambda \delta_{x_1}>_M=\frac{1}{12}b(\lambda).\end{equation}
\end{lemma}

\begin{proof}
$P_\lambda \delta_{x_1}(x)=\varphi(d(x,x_1))$ and $\varphi(n)=O(2^{-n/2})$ by \eqref{5.29}.  It follows easily that the limit, if it exists, is independent of the choice of $x_0$, since if $d(x_0,x_0')=k$ then $B_{n-k}(x_0')\subseteq B_n(x_0)\subseteq B_{n+k}(x_0')$, and the division by $N$ in \eqref{5.28} makes the difference go to zero as $N\to\infty$.  We will prove the existence of the limit by computing \eqref{5.32} with $x_0=x_1$.

Note that there are exactly $3\cdot2^{n-1}$ points $x$ with $d(x,x_0)=n$ for $n\geq1$, and we can ignore the point $x=x_1$ in computing the limit.  Thus
\begin{equation*}
<P_\lambda \delta_{x_1},P_\lambda \delta_{x_1}>_M=\displaystyle\lim_{N\to\infty}\frac{3}{2N}\Sum_{n=1}^N2^n\varphi(n)^2=\frac{1}{12}b(\lambda) \end{equation*}
by Lemma \ref{lemV2}.
\end{proof}

\begin{lemma}\label{lemV4}
Suppose $d(x_1,x_2)=k$ and $\lambda$ is in the interior of $\Sigma$.  Then $<P_\lambda \delta_{x_1},P_\lambda \delta_{x_2}>_M$ exists and is independent of the base point $x_0$, and 
\begin{equation}\label{5.33}
<P_\lambda \delta_{x_1},P_\lambda \delta_{x_2}>_M=\frac{1}{12}b(\lambda)\varphi(k).
\end{equation}
\end{lemma}
\begin{proof}
The proof of independence of the base point is the same as in Lemma \ref{lemV3}, so we compute the limit for $x_0=x_1$.  Except for a few points when $n$ is small that don't enter into the limit, we may partition the points with $d(x,x_1)=n$ as follows:

$2^n$ points with $d(x,x_2)=n+k$,

$2^{n-j-1}$ points with $d(x,x_2)=n+k-2j$ for $1\leq j \leq k-1$,

$2^{n-k}$ points with $d(x,x_k)=n-k$.

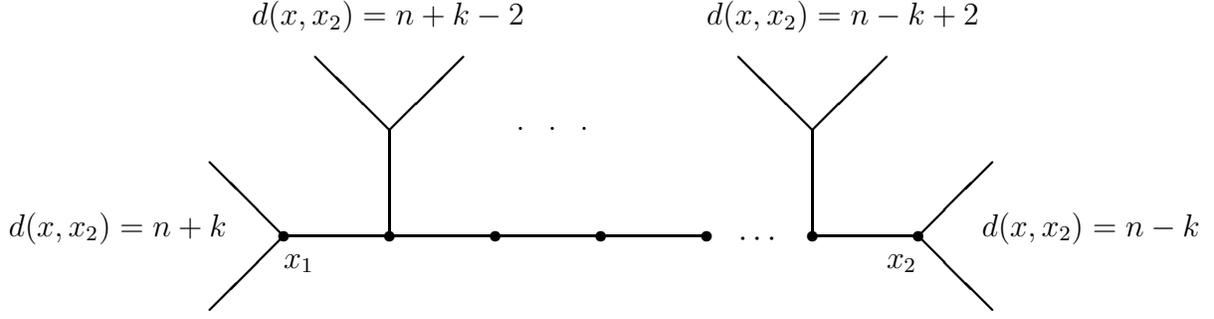
\begin{figure}[htb]
\begin{center}
\begin{picture}(150,120)(-75,-20)\thicklines
\setlength{\unitlength}{40pt}
\put(-3,0){\line(1,0){4}}
\put(2,0){\line(1,0){1}}
\put(1.3,-.1){$\cdots$}
\multiput(-.8,1)(.3,0){3}{$.$}
\put(3,0){\line(1,1){.7}}
\put(3,0){\line(1,-1){.7}}
\put(2,0){\line(0,1){1}}
\put(-2,0){\line(0,1){1}}
\put(-3,0){\line(-1,1){.7}}
\put(-3,0){\line(-1,-1){.7}}
\put(-2,1){\line(-1,1){.7}}
\put(-2,1){\line(1,1){.7}}
\put(2,1){\line(-1,1){.7}}
\put(2,1){\line(1,1){.7}}
\multiput(-3,0)(1,0){5}{\circle*{.1}}
\multiput(2,0)(1,0){2}{\circle*{.1}}
\put(3.6,0){$d(x,x_2)=n-k$}
\put(-5.6,0){$d(x,x_2)=n+k$}
\put(-3.3,2){$d(x,x_2)=n+k-2$}
\put(1,2){$d(x,x_2)=n-k+2$}
\put(-3,-.3){$x_1$}
\put(2.7,-.3){$x_2$}
\end{picture}
\end{center}
\caption{Partition of points $x$ with $d(x,x_1)=n$.}
\end{figure}

This implies
\begin{eqnarray}
&<P_\lambda \delta_{x_1},P_\lambda \delta_{x_2}>_M\hspace{12cm}\notag\\
&=\displaystyle\lim_{N\to\infty}\frac{1}{N}\Sum_{n=1}^N \left(2^n\varphi(n)\varphi(n+k)+\frac{1}{2}\Sum_{j=1}^{k-1} 2^{n-j}\varphi(n)\varphi(n+k-2j)+2^{n-k}\varphi(n)\varphi(n-k)\right)\notag\\
&=\frac{1}{18}b(\lambda)2^{-k/2}\left(\cos(kt\log2)+\frac{1}{2}\Sum_{j=1}^{k-1}\cos(k-2j)t\log2+\cos(kt\log2)\right)\hspace{2cm}\notag
\end{eqnarray}
by Lemma \ref{lemV2}.

However, the trigonometric identity $\sin(a)\Sum_{j=0}^{k-1}\cos(k-2j)a=\sin(ka)\cos(a)$ implies
\begin{eqnarray}
&2\cos(kt\log2)+\frac{1}{2}\Sum_{j=1}^{k-1}\cos(k-2j)t\log2  \notag \\
&=\frac{3}{2}\cos(kt\log2)+\frac{1}{2}\Sum_{j=0}^{k-1}\cos(k-2j)t\log2 \notag \\
&=\frac{3}{2}\left( \cos(kt\log2)+\frac{1}{3} \frac{\sin(kt\log2)}{\tan(t\log2)}\right)=\frac{3}{2} \varphi(k)2^{k/2} \notag
\end{eqnarray}
by Lemma \ref{lemV1}, which implies \eqref{5.33}.
\end{proof}

\begin{theorem}\label{thmV5}
Suppose $f$ has finite support.  Then
\begin{equation} \label{5.34}
<P_\lambda f, f>=12b(\lambda)^{-1}<P_\lambda f, P_\lambda f>_M. \end{equation}
\end{theorem}
\begin{proof}
Since $<P_\lambda \delta_{x_1}, \delta_{x_2}>=\varphi(d(x_1,x_2))$ we can rewrite \eqref{5.33} as $$<P_\lambda \delta_{x_1},\delta_{x_2}>=12b(\lambda)^{-1}<P_\lambda \delta_{x_1},P_\lambda \delta_{x_1}>_M$$ and \eqref{5.34} follows by linearity. \end{proof}
\begin{corollary}\label{corV6}
For $f\in\ell^2(\G)$, for $\mu$ a.e. $\lambda$, $<P_\lambda f,P_\lambda f>_M$ exists, and 
\begin{equation}\label{5.35}
||f||^2_{\ell^2(\G)}=\int_\Sigma<P_\lambda f, P_\lambda f>_M 12b(\lambda)^{-1}d\mu(\lambda).
\end{equation}
\end{corollary}
\begin{proof}
For $f$ of finite support, \eqref{5.35} follows from \eqref{5.34} and \eqref{4.16}.  It then follows for $f\in \ell^2(\G)$ by routine limiting arguments.
\end{proof}

To complete the solution of problem $(c)$ for this example we need to transfer the result from $\G$ to $\Go$.  Define the modified mean inner product on $\Go$ by \eqref{5.28} again, where $f$ and $g$ are functions on $\Go$ and $x$ and $x_0$ vary in $\Go$.

\begin{lemma}\label{lemV7}
For any integers $k$ and $j$,
\begin{eqnarray} \label{5.36}
\displaystyle\lim_{N\to\infty} \frac{1}{N} \Sum_{n=1}^N 2^{n+\frac{k}{2}}\psi(n)\psi(n+k)&=\displaystyle\lim_{N\to\infty} \frac{1}{N} \Sum_{n=1}^N 2^{n+j+\frac{k}{2}}\psi(n+j)\psi(n+j+k)\notag \\
&=\frac{(6-\lambda)^2}{36}b(\lambda)\cos(kt\log2).
\end{eqnarray}
\end{lemma}
\begin{proof}
As in the proof of Lemma \ref{lemV2}, it is clear that \eqref{5.36} is independent of $j$, so we may take $j=0$.  Since $\psi(k)=2\varphi(k)+\varphi(n-1)+\varphi(n+1)$ we may reduce \eqref{5.36} to \eqref{5.31} as follows:
\begin{eqnarray}
\displaystyle\lim_{N\to\infty}&\frac{1}{N}\Sum_{n=1}^N 2^{n+\frac{k}{2}}\psi(n)\psi(n+k)\hspace{9cm}& \notag \\
=\displaystyle\lim_{N\to\infty}&\hspace{-13cm}\frac{1}{N}\Sum_{n=1}^N & \hspace{-13cm} 2^{n+\frac{k}{2}} \left(2\varphi(n)+\varphi(n{-}1)+\varphi(n{+}1)\right)\left(2\varphi(n{+}k)+\varphi(n{+}k{-}1)+\varphi(n{+}k{+}1)\right) \notag\\
&\hspace{-4cm}=\frac{b(\lambda)}{18}[(4+2+\frac{1}{2})\cos(kt\log2)+2(\sqrt{2}+\frac{1}{\sqrt{2}})(\log(k+1)t\log2&\hspace{-3cm}+\log(k-1)t\log2)\notag\\
&&\hspace{-6cm}+\cos(k+2)t\log2+\cos(k-2)t\log2]\notag\\
=&\frac{b(\lambda)}{18} \cos kt\log2[(4+2+\frac{1}{2})+\psi(\sqrt{2}+\frac{1}{\sqrt{2}})\cos t\log2+2\cos2t\log2]&\notag\\
&\hspace{-4cm}=\frac{b(\lambda)}{18}\cos kt\log2(\frac{3}{\sqrt{2}}+2\cos t\log2)^2& \notag
\end{eqnarray}
and \eqref{5.36} follows since $\frac{3}{\sqrt{2}}+2\cos t\log2 =\frac{(6-\lambda)}{\sqrt{2}}$.
\end{proof}  

\begin{lemma}\label{lemV8}
For any $\lambda$ in the interior of $\Sigma$ and $x_1\in \Go$, $<\tilde{P_\lambda}\delta_{x_1}, \tilde{P_\lambda}\delta_{x_1}>_M$ exists and is independent of the base point $x_0$, and
\begin{equation} \label{5.37}
<\tilde{P_\lambda}\delta_{x_1},\tilde{P_\lambda}\delta_{x_1}>_M=\frac{b(\lambda)}{162}.
\end{equation}
\end{lemma}
\begin{proof}
The proof that the limit is independent of the base point is the same as in Lemma \ref{lemV3}, so we compute \eqref{5.36} with $x_0=x_1$.  Note that for $n\geq1$ there are exactly $4\cdot 2^{n-1}$ points $x$ in $V_0$ with $d(x,x_1)=n$.  
For such points $\tilde{P_\lambda}\delta_{x_1}(x)=\frac{1}{6-\lambda}\frac{1}{3}\psi(n)=\frac{1}{6-\lambda}\frac{1}{3}\left(2\varphi(n)+\varphi(n{-}1)+\varphi(n{+}1)\right)$
and so 
\begin{equation*}
<\tilde{P_\lambda}\delta_{x_1},\tilde{P_\lambda}\delta_{x_1}>_M=\frac{1}{(6-\lambda)^2}\cdot\frac{2}{9}\displaystyle\lim_{N\to\infty} \frac{1}{N} \Sum_{n=1}^\infty 2^n\left(2\varphi(n)+\varphi(n-1)+\varphi(n+1)\right)^2  \end{equation*}
and \eqref{5.37} follows from \eqref{5.36}.
\end{proof}

\begin{lemma}\label{lemV9}
Suppose $d(x_1,x_2)=k$ and $\lambda$ is in the interior of $\Sigma$.  Then $<\tilde{P_\lambda}\delta_{x_1},\tilde{P_\lambda}\delta_{x_2}>_M$ exists and is independent of the base point, and 
\begin{equation}
<\tilde{P_\lambda}\delta_{x_1},\tilde{P_\lambda}\delta_{x_2}>_M=\frac{b(\lambda)}{36}\cdot\frac{1}{3(6-\lambda)} \psi
(k). \end{equation}
\end{lemma}
\begin{proof}
As before we can take the base point $x_0=x_1$.  For $n>k$ we can sort the $2^{n+1}$ points $x$ with $d(x,x_1)=n$ as follows:

$2^n$ points with $d(x,x_2)=n+k$,

$2^{n-j}$ points with $d(x,x_2)=n+k-2j+1$ for $1\leq j \leq k$, and

$2^{n-k}$ points with $d(x,x_2)=n-k$.

\begin{figure}[htb]
\begin{center}
\begin{picture}(150,150)(-92,-70)\thicklines
\setlength{\unitlength}{40pt}
\put(0,0){\line(-1,0){4}}
\put(0,0){\line(-1,1){1}}
\put(-2,0){\line(1,1){1}}
\put(-2,0){\line(-1,-1){1}}
\put(-4,0){\line(1,-1){1}}
\multiput(.2,0)(.3,0){3}{$.$}
\put(-6.2,-.2){$d(x,x_2)=n+k$}
\put(-2.8,1.3){$d(x,x_2)=n+k-3$}
\put(-4.3,-1.4){$d(x,x_2)=n+k-1$}
\put(-4.1,.2){$x_1$}
\qbezier(-2,1)(-1,.5)(0,1)
\qbezier(-4,-1)(-3,-.5)(-2,-1)
\qbezier(-4.3,1)(-3.1,0)(-4.3,-1)
\multiput(-4,0)(2,0){3}{\circle*{.1}}
\put(.7,-1.4){$d(x,x_2)=n-k+1$}
\put(3.3,-.2){$d(x,x_2)=n-k$}
\put(3.2,.2){$x_2$}
\put(1.2,0){\line(1,0){2}}
\multiput(1.2,0)(2,0){2}{\circle*{.1}}
\put(3.2,0){\line(-1,-1){1}}
\put(1.2,0){\line(1,-1){1}}
\qbezier(1.2,-1)(2.2,-.5)(3.2,-1)
\qbezier(3.4,1)(2.5,0)(3.4,-1)
\end{picture}
\end{center}
\caption{Partition of points $x$ with $d(x,x_1)=n$.}
\end{figure}

Thus we have
\begin{eqnarray}
<\tilde{P_\lambda}\delta_{x_1},\tilde{P_\lambda}\delta_{x_2}>_M=\frac{1}{(6-\lambda)^2}\cdot\frac{1}{9} \displaystyle\lim_{N\to\infty} \frac{1}{N} \Sum_{n=1}^N & \notag \\
&\hspace{-6cm}\psi(n) \left(2^n \psi(n+k)+\Sum_{j=1}^k 2^{n-j}\psi(n+k-2j+1)+2^{n-k}\psi(n-k)\right) \notag \\
&\hspace{-6cm}=\frac{b(\lambda)}{9\cdot36}2^{-k/2}\left[\cos(kt\log2)+\frac{1}{\sqrt{2}}\Sum_{j=1}^k \cos(k-2j+1) t\log2 +\cos(kt\log2)\right] \notag
\end{eqnarray}
by \eqref{5.36}.

To complete the proof we need to show
\begin{eqnarray}
\frac{2^{-k/2}}{9}&\left[2\cos(kt\log2)+\frac{1}{\sqrt{2}}\Sum_{j=1}^k\cos(k-2j+1)t\log2 \right]\notag \\
&=\frac{1}{3(6-\lambda)}(2\varphi(k)+\varphi(k-1)+\varphi(k+1)). \notag
\end{eqnarray}
As we saw in the proof of Lemma \ref{lemV4}, $\varphi(k)=\frac{2}{3} 2^{-k/2}(2\cos(kt\log2)+\frac{1}{2}\Sum_{j=1}^{k-1}\cos(k-2j)t\log2)$ so
\begin{eqnarray}
2\varphi(k)&\hspace{-1cm}+\varphi(k-1)+\varphi(k+1)=\frac{2}{3} 2^{-k/2}\left(4\cos(kt\log{2}+\Sum_{j=1}^{k-1} \cos(k-2j)\log2\right. \notag \\
&+ 2 \sqrt{2} \cos(k-1)t\log2 +\frac{\sqrt{2}}{2} \Sum_{j=1}^{k-2}\cos(k-2j-1)t\log2+\sqrt{2}\cos(k+1)t\log2 \notag \\
&\left.+\frac{1}{2\sqrt{2}} \Sum_{j=1}^k \cos(k-2j+1)t\log2 \right) \notag
\end{eqnarray}
and the result follows by standard trigonometric identities. 
\end{proof}

\begin{theorem}\label{thmV8}
Suppose $F$ has finite support on $\Go$.  Then
\begin{equation}\label{5.39}
<\tilde{P_\lambda}F,F>=36b(\lambda)^{-1}<\tilde{P_\lambda}F,\tilde{P_\lambda}F>_M. 
\end{equation}
\end{theorem}
\begin{proof}
Since $<\tilde{P_\lambda}\delta_{x_1},\delta_{x_2}> =\frac{1}{3(6-\lambda)}\psi(d(x_1,x_2))$ we can rewrite \eqref{5.37} as $<\tilde{P_\lambda}\delta_{x_1},\delta_{x_2}> =36b(\lambda)^{-1}<\tilde{P_\lambda}\delta_{x_1},\tilde{P_\lambda}\delta_{x_1}>_M$ and \eqref{5.39} follows by linearity.
\end{proof}

\begin{corollary} \label{corV9}
For $F\in\ell^2(\Go)$, for $\mu$-a.e. $\lambda$ in $\Sigma$, $<\tilde{P_\lambda}F,\tilde{P_\lambda}F>_M$ exists, and
\begin{equation*}
||F||^2_{\ell^2(\Go)}=||\tilde{P_6}F||^2_2+\int_\Sigma <\tilde{P_\lambda}F,\tilde{P_\lambda}F>_M 36 b(\lambda)^{-1} d\mu(\lambda).
\end{equation*}
\end{corollary}
\begin{proof}
Same as for Corollary \ref{corV6}. \end{proof}

We end this section with a description of 5-series eigenfunctions on the graph $\G_1$ (note there are no 5-eigenfunctions on the graph $\G_0$). One can easily see that on  $\G_1$ there are no finitely supported 5-eigenfunction, there are no radially symmetric 5-eigenfunctions, and that 5-eigenfunctions do not correspond to cycles. 
\begin{figure}[htb]
\def\trigup{\put(0, 0){\line(1, 0){6}}\put(0, 0){\line(3, 5){3}}\put(6, 0){\line(-3, 5){3}}}
\def\trigdown{\put(0, 0){\line(1, 0){6}}\put(0, 0){\line(3, -5){3}}\put(6, 0){\line(-3, -5){3}}}
\def\trigright{\put(0, 0){\line(5, 3){5}}\put(0, 0){\line(5, -3){5}}\put(5, 3){\line(0, -1){6}}}
\def\trigleft{\put(0, 0){\line(-5, 3){5}}\put(0, 0){\line(-5, -3){5}}\put(-5, 3){\line(0, -1){6}}}
\def\tu#1#2#3{\setlength{\unitlength}{2.5pt}\thinlines\put(-3,-10){\trigup}\put(0,-5){\trigup}\put(3,-10){\trigup}\put(0,-5){\llap{\SMALL{#1}}}\put(6,-5){\rlap{\SMALL{#2}}}\put(2,-13){\rlap{\SMALL{#3}}}}
\def\td#1#2#3{\setlength{\unitlength}{2.5pt}\thinlines\put(-3,10){\trigdown}\put(0,5){\trigdown}\put(3,10){\trigdown}\put(0,3){\llap{\SMALL{#1}}}\put(6,3){\rlap{\SMALL{#2}}}\put(2,11){\rlap{\SMALL{#3}}}}
\def\tr#1#2#3{\setlength{\unitlength}{2.5pt}\thinlines\put(0,0){\trigright}\put(5,3){\trigright}\put(5,-3){\trigright}\put(5,3){\llap{\SMALL{#1}}}\put(5,-6){\llap{\SMALL{#2}}}\put(11,0){\rlap{\SMALL{#3}}}}
\def\tl#1#2#3{\setlength{\unitlength}{2.5pt}\thinlines\put(0,0){\trigleft}\put(-5,3){\trigleft}\put(-5,-3){\trigleft}\put(-5,3){\rlap{\SMALL{#1}}}\put(-5,-6){\rlap{\SMALL{#2}}}\put(-11,0){\llap{\SMALL{#3}}}}
\centering
\begin{picture}(120, 165)(-60, -80)
\setlength{\unitlength}{2.5pt}
\put(0,0){\tu{4}{4}{-8}\td{-4}{-4}{8}}
\put(12,20){\tu{-2}{4}{-2}\td{-1}{-1}{2}}
\put(-12,20){\tu{4}{-2}{-2}\td{-1}{-1}{2}}
\put(12,-20){\tu{1}{1}{-2}\td{2}{-4}{2}}
\put(-12,-20){\tu{1}{1}{-2}\td{-4}{2}{2}}
\put(-15,10){\tl{-1}{-1}{2}}
\put(-15,-10){\tl{1}{1}{-2}}
\put(21,-10){\tr{1}{1}{-2}}
\put(21,10){\tr{-1}{-1}{2}}
\end{picture}\caption{A part of $\G_1$ with a 5-eigenfunction (values not shown are equal to zero).}\label{fig-5}
\end{figure}
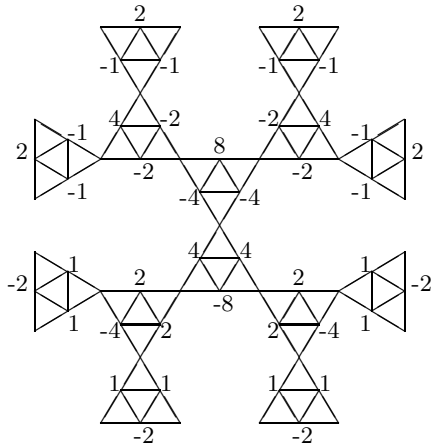
By by an argument similar to Theorem~\ref{thm5.1} one can show that eigenfunctions in Figure~\ref{fig-5} (with their translations, rotations and reflections), are complete in the eigenspace $E_5$  on  $\G_1$. We do not give an explicit formula for the  5-eigenprojections on $\G_n$. One can see that for each $n>1$ there are eigenfunctions on $\G_n$ that resemble those in  Figure~\ref{fig-5}, 
 and also
finitely supported 5-eigenfunctions (see Remark~\ref{rem-K}).  

\section{Periodic Fractafolds}\label{sec6}

\begin{remark}\label{rem-K} Note  that  on a periodic graph, linear combinations of compactly
supported eigenfunctions are dense in an eigenspace  (see  \cite[Theorem~8]{Kuchment05}, \cite{Kuchment93} and \cite[Lemma~3.5]{KuchmentPost}). 

The computation of compactly 
supported 5- and 6- series eigenfunctions is discussed in detail in \cite{St03,T98}, and the eigenfunctions with compact support are complete in the corresponding eigenspaces. 

In particular, \cite{St03,T98} show that any   6-series finitely supported eigenfunction on $\G_{n+1}$ is the  continuation of any finitely supported function on $\G_n$, and the corresponding continuous eigenfunction on the \Sif\ $\fF$ can be computed using the eigenfunction extension map on fractafolds (see Subsection~\ref{subsec-eem}). 
Similarly, any   5-series finitely supported eigenfunction on $\G_{n+1}$ can be described by a cycle of triangles (homology) in  $\G_n$, and the corresponding continuous eigenfunction on the \Sif\ $\fF$ is computed using the eigenfunction extension map on fractafolds. 
 \end{remark}

\begin{example}{The Ladder Fractafold.}  {\normalfont Here $\G$ is the ladder graph consisting of two copies of $\mathbb{Z}$, $\{a_k\}$ and $\{b_k\}$ with $a_k \sim b_k$
\begin{figure}[htb]
\centering\def\q{\includegraphics[height=52pt,width=30pt]{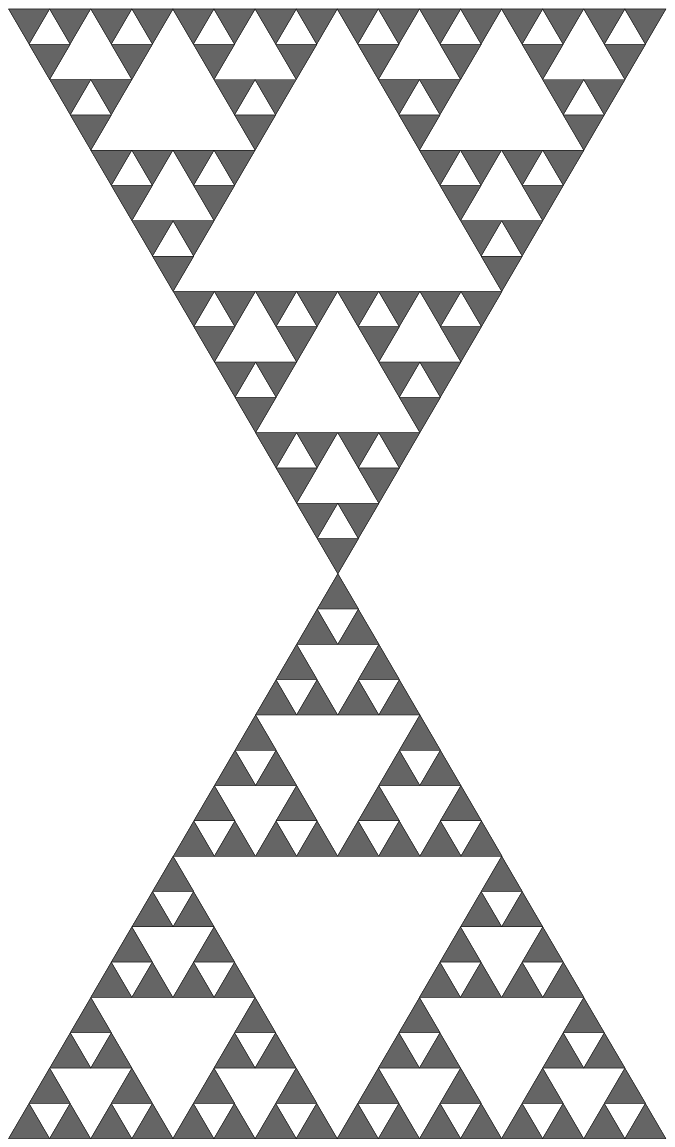}}
\begin{picture}(240, 60)(0, 0)
\put(0,0){\multiput(0, 0)(29.55,0){8}{\q}}
\end{picture}\caption{A part of the infinite Ladder  \Sif.}\label{fig-h-ladder}
\end{figure}
\begin{figure}[hbt]
\begin{center}
\begin{picture}(120,45)(0,25)\thicklines
\setlength{\unitlength}{12pt}
\put(0,3){\line(1,0){12}}
\put(0,6){\line(1,0){12}}
\multiput(3,3)(3,0){3}{\line(0,1){3}}
\put(3,2){$b_{-1}$}
\put(6,2){$b_{0}$}
\put(9,2){$b_{1}$}
\put(3,6.5){$a_{-1}$}
\put(6,6.5){$a_{0}$}
\put(9,6.5){$a_{1}$}
\put(12,4.5){$\dots$}
\put(-1,4.5){$\dots$}
\end{picture}
\end{center}
\caption{$\G$ graph for the Ladder Fractafold}\label{fig1}\end{figure}
and $\Go$ consisting of three copies of $\mathbb{Z}$, $\{x_{k+\frac{1}{2}}\}, \{w_k\}, \{y_{k+\frac{1}{2}}\}$ with $w_k$ joined to $x_{k-\frac{1}{2}},x_{k+\frac{1}{2}},y_{k-\frac{1}{2}}$, and $y_{k+\frac{1}{2}}$, 
\begin{figure}[hbt]
\begin{center}
\begin{picture}(120,90)(5,40)\thicklines
\setlength{\unitlength}{15pt}
\put(0,3){\line(1,0){9}}
\put(0,7.5){\line(1,0){9}}
\multiput(0,3)(3,0){3}{\line(2,3){3}}
\multiput(3,3)(3,0){3}{\line(-2,3){3}}
\put(0,2.3){$y_{-\frac32}$}
\put(3,2.3){$y_{-\frac12}$}
\put(6,2.3){$y_{\frac12}$}
\put(9,2.3){$y_{\frac32}$}
\put(0,8.03){$x_{-\frac32}$}
\put(3,8.03){$x_{-\frac12}$}
\put(6,8.03){$x_{\frac12}$}
\put(9,8.03){$x_{\frac32}$}
\put(1.7,5){$w_{-1}$}
\put(4.7,5){$w_{0}$}
\put(7.7,5){$w_{1}$}
\put(9,5.5){$\dots$}
\put(-1,5.5){$\dots$}
\end{picture}
\end{center}
\caption{$\Go$ graph for the Ladder Fractafold}\label{fig2}\end{figure}
where $x_{k+\frac{1}{2}}$ is the edge $[a_k, a_{k+1}]$, $x_{y+\frac{1}{2}}$ is the edge $[b_k,b_{k+1}]$ and $w_k$ is the edge $[a_k,b_k]$.

It is easy to see that the spectrum of $-\D_\G$ is $[0,6]$, with the even functions $\varphi_\theta(a_k)=\varphi_\theta(b_k)=\cos k\theta$ or $\sin k\theta$, $0\leq\theta\leq\pi$ corresponding to $\lambda=2-2\cos\theta$ in $[0,4]$ and the odd functions $\psi_\theta(a_k)=-\psi_\theta(b_k)=\cos k\theta$ or $\sin k\theta$, $0\leq\theta\leq\pi$ corresponding to $\lambda=4-2\cos\theta$ in $[2,6]$.

These transfer to eigenfunctions of $-\D_\Go$
\begin{eqnarray*}
&\tilde{\varphi_\theta}(x_{k+\frac{1}{2}})=\tilde{\varphi_\theta}(y_{k+\frac{1}{2}})=\cos(k+\frac{1}{2})\theta\cos\frac{1}{2}\theta \mbox{ or }  \sin(k+\frac{1}{2})\theta\cos\frac{1}{2}\theta\\
&\tilde{\varphi_\theta}(w_k)=\cos k\theta \mbox{ or } \sin k\theta \mbox{ and}\\
&\tilde{\psi_\theta}(x_{k+\frac{1}{2}})=-\tilde{\psi_\theta}(y_{k+\frac{1}{2}})=\cos(k+\frac{1}{2})\theta \mbox{ or } \sin(k+\frac{1}{2})\theta\\
&\tilde{\psi_\theta}(w_k)=0
\end{eqnarray*}
with the same eigenvalues.  It is also easy to see that there are no $\ell^2(\Go)$ eigenfunctions corresponding to $\lambda=6$ (or for any $\lambda$ value whatsoever).  Thus $-\D_\Go$ has absolutely continuous spectrum $[0,6]$ with multiplicity $2$ in $[0,2]$ and $[4,6]$ and multiplicity $4$ in $[2,4]$.
}\end{example}

\begin{example}{The Honeycomb Fractafold.} {\normalfont Here $\G$ is the hexagonal graph consisting of the triangular lattice $\mathcal{L}$ generated by $(1,0)$ and $(\frac{1}{2},\frac{\sqrt{3}}{2})$ and the displaced lattice $\mathcal{L}+(\frac{1}{2},\frac{\sqrt{3}}{6})$.  We denote by $a(j,k)$ the points $j(1,0)+k(\frac{1}{2},\frac{\sqrt{3}}{2})$ of $\mathcal{L}$ and by $ b(j,k)$ 
the points $a(j,k)+(\frac{1}{2},\frac{\sqrt{3}}{6})$ of the displaced lattice, with edges $a(j,k)\sim b(j,k)$, $a(j,k)\sim b(j-1,k)$ and $a(j,k)\sim b(j,k-1)$.
\begin{figure}[htb]
\centering\def\q{\includegraphics[height=52pt,width=30pt]{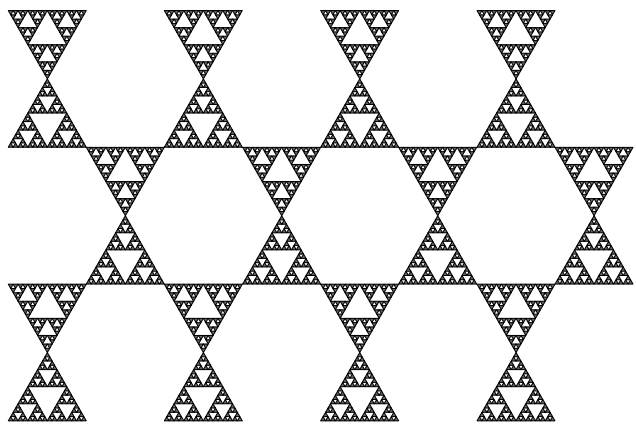}}
\begin{picture}(240, 160)(0, 0)
\put(29.55,51.8){\multiput(0, 0)(59.1,0){4}{\q}}
\multiput(0, 0)(0,103.6){2}{\multiput(0, 0)(59.1,0){4}{\q}}
\end{picture}\caption{A part of the infinite periodic \Sif\ based on the hexagonal (honeycomb) lattice.}\label{fig-h-lattice}
\end{figure}
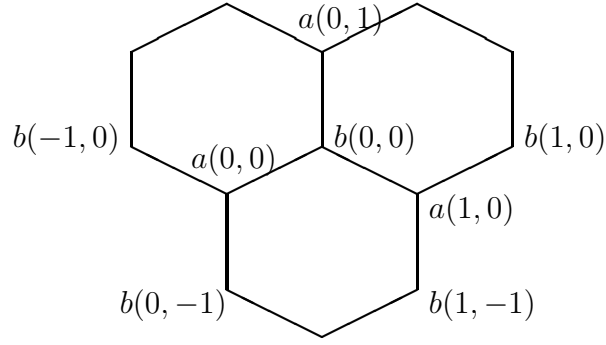
\begin{figure}[hbt]
\begin{center}
\begin{picture}(200,120)(-100,-65)\thicklines
\setlength{\unitlength}{9pt}
\put(-8,4){\line(2,1){4}}
\put(-4,-2){\line(2,1){4}}
\put(0,4){\line(2,1){4}}
\put(4,-2){\line(2,1){4}}
\put(0,-8){\line(2,1){4}}
\put(0,4){\line(-2,1){4}}
\put(8,4){\line(-2,1){4}}
\put(4,-2){\line(-2,1){4}}
\put(-4,-2){\line(-2,1){4}}
\put(0,-8){\line(-2,1){4}}
\put(0,0){\line(0,1){4}}
\put(-8,0){\line(0,1){4}}
\put(8,0){\line(0,1){4}}
\put(4,-6){\line(0,1){4}}
\put(-4,-6){\line(0,1){4}}
\put(-1,5){$a(0,1)$}
\put(8.5,0){$b(1,0)$}
\put(4.5,-3){$a(1,0)$}
\put(4.5,-7){$b(1,-1)$}
\put(-8.5,-7){$b(0,-1)$}
\put(-5.5,-.9){$a(0,0)$}
\put(-13,0){$b(-1,0)$}
\put(.5,0){$b(0,0)$}
\end{picture}
\end{center}
\caption{A part of the Hexagonal graph}\label{fig3}\end{figure}
The eigenfunctions of $-\D_\G$ will have the form 
\begin{eqnarray*}
& \varphi_{u,v}(a(j,k))=e^{2\pi i (ju+kv)}\\
& \varphi_{u,v}(b(j,k))=\gamma e^{2\pi i (ju+kv)}
\end{eqnarray*}
where $(u,v)\in[0,1]\times[0,1]$ and $\gamma$ depends on $u,v$.  Let $1 + e^{2\pi i u}+e^{2\pi i v}=re^{i \theta}$ in polar coordinates (so $r$ and $\theta$ are functions of $u,v$).  Note that $0\leq r\leq3$.  Then the eigenvalue equation requires $\gamma^2=e^{2 i \theta}$ or $\gamma=\pm e^{i\theta}$ with corresponding eigenvalues $\lambda=3\mp r$ (so the choice $\pm$ yields the intervals $[0,3]$ and $[3,6]$ in $spect(-\D_\G)$).

We can write the explicit spectral resolution as follows.  For $f\in\ell^2(\G)$ define $$\hat{f_a}(u,v)=\displaystyle\sum_j\displaystyle\sum_k e^{-2\pi i(ju+kv)}f(a(j,u))$$ and $$\hat{f_b}(u,v)=\displaystyle\sum_j\displaystyle\sum_k e^{-2\pi i(ju+kv)}f(b(j,u)).$$

We can invert these so that
\begin{eqnarray*}
& 
\displaystyle{{f(a(j,k)) \brace f(b(j,k))}} =&\int_0^1\int_0^1 {1 \brace e^{i \theta}} e^{2\pi i(ju+kv)}\frac{1}{2}(\hat{f_a}(u,v)+e^{-i \theta} \hat{f_b}(u,v))du dv\\
&&+\int_0^1\int_0^1 {1 \brace -e^{i \theta}} 
e^{2\pi i(ju+kv)}
\frac{1}{2}(\hat{f_a}(u,v)-e^{-i \theta} \hat{f_b}(u,v))du dv.
\end{eqnarray*}
Define $\lambda_\pm(u,v)$ by $$\lambda_\pm(u,v)=3\mp \sqrt{3+2 \cos 2\pi u + 2\cos2\pi v + 2\cos 2 \pi(u-v)}.$$
For $0\leq\lambda\leq3$ we define $u_\theta$ and $v_\theta$ by solving 
$\lambda_+(u,v)=\lambda$, and similarly for $3\leq\lambda\leq6$ we solve $\lambda_-(u,v)=\lambda$.  We then define
\begin{equation}
{P_\lambda f(a(j,k)) \brace P_\lambda f(b(j,k))} = \int_0^{2\pi} {1 \brace \pm e^{i \theta}}e^{2\pi i (j u_\theta +k v_\theta)}\frac{1}{2}(\hat{f_a}(u_\theta,v_\theta) \pm e^{-i \theta} \hat{f_b}(u_\theta,v_\theta))\left|\frac{\partial(u_\theta,v_\theta)}{\partial(\lambda,\theta)}\right|d\theta
\end{equation}
to obtain $f=\int_0^6 P_\lambda f d\lambda$ with $-\D_\G P_\lambda f= \lambda P_\lambda f$.  This solves problem $(a)$.

To solve problem $(b)$ we identify the space $E_6$ in $\ell^2(\Go)$.

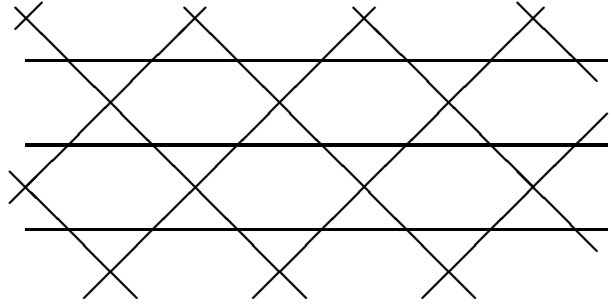
\begin{figure}[hbt]
\begin{center}
\begin{picture}(120,110)(-30,-40)\thicklines
\setlength{\unitlength}{4pt}
\put(-24,0){\line(1,1){17}}
\put(-16,-8){\line(1,1){25}}
\put(0,-8){\line(1,1){25}}
\put(16,-8){\line(1,1){15}}
\put(0,-8){\line(-1,1){25}}
\put(16,-8){\line(-1,1){25}}
\put(24,0){\line(-1,1){17}}
\put(-24,4){\line(1,0){55}}
\put(-24,-4){\line(1,0){55}}
\put(-24,12){\line(1,0){55}}
\put(30,10){\line(-1,1){7.5}}
\put(30,-6){\line(-1,1){7.5}}
\put(-16,-8){\line(-1,1){9.5}}
\put(-23,1){\line(-1,-1){2.5}}
\put(-25,15){\line(1,1){2.5}}
\multiput(-16,-8)(16,0){3}{\put(0,0){\line(-1,-1){2.5}}\put(0,0){\line(1,-1){2.5}}}
\end{picture}
\end{center}
\caption{A part of the graph $\Go$ for the Honeycomb Fractafold}\label{fig4}\end{figure}

We may regard $\Go$ as an infinite union of hexagons, each vertex belonging to exactly two hexagons.  For any fixed hexagon $H$, define $\psi_H$ to alternate values $\pm 1$ around the verticies of $H$, and to vanish elsewhere.  It is easy to see that $\psi_H$ in is $E_6$.  If $\{H_j\}$ is an enumeration of all the hexagons in $\Go$ then $\sum c_j \psi_{H_j}$ (finite sum) is in $E_6$.
}\end{example}  

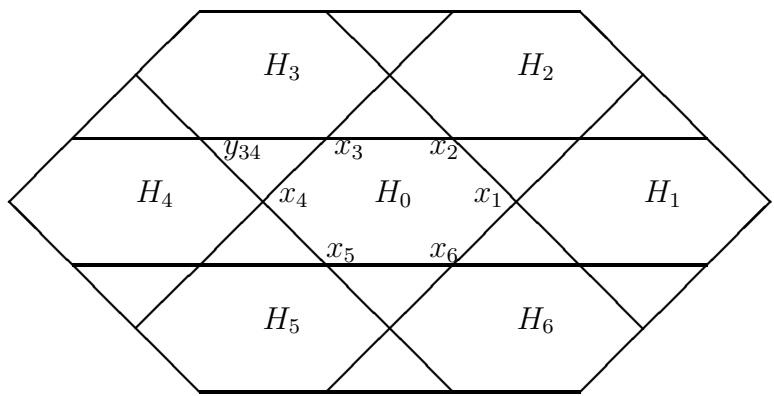
\begin{figure}[hbt]
\begin{center}
\begin{picture}(150,150)(-70,-70)\thicklines
\setlength{\unitlength}{6pt}
\put(-24,0){\line(1,1){12}}
\put(-16,-8){\line(1,1){20}}
\put(-4,-12){\line(1,1){20}}
\put(12,-12){\line(1,1){12}}
\put(-12,-12){\line(-1,1){12}}
\put(4,-12){\line(-1,1){20}}
\put(16,-8){\line(-1,1){20}}
\put(24,0){\line(-1,1){12}}
\put(-20,4){\line(1,0){40}}
\put(-20,-4){\line(1,0){40}}
\put(-12,12){\line(1,0){24}}
\put(-12,-12){\line(1,0){24}}
\put(-1,0){$H_0$}
\put(-8,8){$H_3$}
\put(8,8){$H_2$}
\put(16,0){$H_1$}
\put(-16,0){$H_4$}
\put(-8,-8){$H_5$}
\put(8,-8){$H_6$}
\put(-3.5,3){$x_3$}
\put(2.5,3){$x_2$}
\put(5.3,0){$x_1$}
\put(-7,0){$x_4$}
\put(-4,-3.5){$x_5$}
\put(2.5,-3.5){$x_6$}
\put(-10.5,3){$y_{34}$}
\end{picture}
\end{center}
\caption{Labels of hexagons and points}\label{fig5}\end{figure}

\begin{lemma}
Suppose $u\in E_6$ has compact support.  Then $u=\sum c_j \psi_{H_j}$ (finite sum). \end{lemma}\begin{proof}
Suppose $supp  (u) \subseteq \displaystyle\bigcup_{j\in A} H_j$  
We will show that there exists $j_0\in A$ and $c_{j_0}$ such that $supp (u-c_{j_0} \psi_{H_{j_0}})\subseteq \displaystyle\bigcup_{j\in A\setminus \{j_0\}}H_j$.  The proof is then completed by induction.

We choose $j$ so that $H_j$ lies in the top row and right-most down-right slanting diagonal of $\cup_{j\in A}H_j$.  In Figure \ref{fig5} above, $j'=0$ and $u$ vanishes on $H_1$, $H_2$, and $H_3$.  So $u(x_1)=0$, $u(x_2)=0$, $u(x_3)=0$.  But $u(x_3)+u(x_4)+u(y_{34})=0$ because $E_6=ker (S_2)$ and $u(y_{34})=0$ since $y_{34}\in H_3$.  
So $u(x_4)=0$.  A similar argument shows $u(x_6)=0$.  The only vertex left in $H_0$ is $x_5$.  By subtracting off $u(x_5) \psi_{H_5}$
we can make $u$ vanish on $H_0$.

We can systematically go across the top row in $supp(u)$ from right to left and remove each hexagon, only changing $u$ on the row below it.  
\begin{figure}[hbt]
\begin{center}
\begin{picture}(60,60)(-30,-25)\thicklines
\setlength{\unitlength}{5pt}
\put(-22,4){\line(1,0){44}}
\put(-22,-4){\line(1,0){44}}
\multiput(-24,0)(16,0){3}{\put(0,0){\line(1,1){6}}\put(0,0){\line(1,-1){6}}}
\multiput(24,0)(-16,0){3}{\put(0,0){\line(-1,1){6}}\put(0,0){\line(-1,-1){6}}}
\put(-1,0){$H_0$}
\put(16,0){$H_1$}
\put(-16,0){$H_4$}
\multiput(-20, -4)(8,0)6{\circle*{1}}
\multiput(-20, 4)(8,0)6{\circle*{1}}
\multiput(-24, 0)(16,0)4{\circle*{1}}
\end{picture}
\end{center}
\caption{A row of hexagons}\label{fig6}\end{figure}
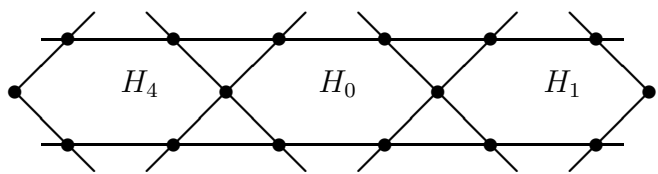
Eventually $u$ will be supported on just one row, and $u(x)=0$ unless $x$ is one of the dotted points in Figure \ref{fig6}.

Let $H_0$ be the right most hexagon.  The $u \mid_{H_1}=0$ implies $u(x_1)=0$ and $u(x_6)=0$.  Considering the triangle below the row we get $u(x_5)=0$.  Considering the triangle above $x_4$ we get $u(x_4)=0$.  So $u\mid_{H_0}=0$.
\end{proof}

\begin{corollary}
A function of compact support is in $E_6$ if and only if $u(x_1)+u(x_2)+u(x_3)=0$ for every triangle $\{x_1,x_2,x_3\}$ in $\Go$.
\end{corollary}
\begin{proof}
The identity clearly holds for each $\psi_H$, hence for all compactly supported functions in $E_6$.  Conversely, every point $x$ in $\Go$ lies in exactly two triangles.  Summing the identity for those two triangles yields the 6-eigenvalue equation at the point $x$.  \end{proof}

The functions $\{\psi_{H_j}\}$ do not form a tight frame, and it seems unlikely that they even form a frame (the 
lower frame bound is doubtful), so they do not seem well suited for describing $\tilde{P_6}$.  We can, however, find an orthonormal basis of $E_6$ that consists of translates of a single function, but we pay the price that the function is not compactly supported.

We change notation to index the hexagons in Figure~\ref{fig4} 
by the lattice $[j,k]=j {0 \brace 1} + k {\frac{1}{2} \brace \frac{\sqrt{3}}{2}}$.  Note that hexagon $H_{[j,k]}$ has six neighbors $H_{[j',k']}$ for $$[j',k']=[j,k]+\{[1,0],[-1,0],[0,1],[0,-1],[1,-1],[-1,1]\}.  $$

To describe a function
\begin{equation}\label{6.2}
F=\Sum_{\mathbb{Z}^2} f([j,k]) \psi_{H_{[j,k]}} \end{equation} 
it suffices to give the discrete Fourier transform $\hat{f}(a,b)$ for $(a,b)\in [0,1]\times [0,1]$ given by
\begin{equation} \label{6.3}
\hat{f}(a,b)=\Sum_{\mathbb{Z}^2} f([j,k]) e^{-2\pi i (aj +bk)}, \end{equation}
for then
\begin{equation}\label{6.4}
f([j,k])=\int_0^1 \int_0^1 e^{2 \pi i (aj+bk)}\hat{f}(a,b) da db. \end{equation}
In fact we will construct $\hat{f}(a,b)$ directly, and then substitute this in \eqref{6.4} and then in \eqref{6.2} to obtain our function in $E_6$.

The basic observation is that each point in $\Go$ lies in exactly two neighboring hexagons, and the values of $\psi_H$ for those two hexagons will be $\pm 1$.  Thus
\begin{equation*}
<F,F>_{\ell^2(\Go)}=\sum |f([j,k])-f([j',k'])|^2 \end{equation*}
for $f$ of the form \eqref{6.2}, where the sum is over all neighboring pairs, and by polarization
\begin{equation}\label{6.5}
<F,G>_{\ell^2(\Go)}=\sum(f([j,k])-f([j',k']) (\overline{g([j,k])-g([j',k'])}) \end{equation}
if $F$ and $G$ are of the form \eqref{6.2}.  Now we substitute \eqref{6.4} in \eqref{6.5} to obtain
\begin{eqnarray}\label{6.6}%
&
<F,G>_{\ell^2(\Go)}=
\hspace{9cm} 
\\
& 
\displaystyle 
\int_0^1 \int_0^1 \Sum_{\mathbb{Z}^2} e^{2\pi i(aj+bk)}\hat{f}(a,b)
[6{-}e^{2 \pi i a}{-}e^{-2 \pi i a}{-}e^{2 \pi i b}{-}e^{-2\pi i b}{-}e^{2\pi i (a-b)}{-}e^{2\pi i (b-a)}] \overline{g([j,k])} da db 
\notag\end{eqnarray}
because of the form of the neighboring relation between $[j,k]$ and $[j',k']$.  But then we can evaluate the sum in \eqref{6.6} using \eqref{6.3} to obtain
\begin{equation}\label{e6.7}
<F,G>_{\ell^2(\Go)}=
\int_0^1\int_0^1 2
\left(3{-}\cos(2\pi a){-}\cos(2\pi b){-}\cos(2\pi(a-b))\right)\hat{f}(a,b)\overline{\hat{f}(a,b)} da db. \end{equation}
\begin{lemma}\label{lem6.5}
The functions $\tau_{p,q}F=\Sum_{\mathbb{Z}^2}f([j,k]+[p,q]) \psi_{H_{[j,k]}}$ form an orthonormal basis of $E_6$ for $[p,q]\in \mathbb{Z}^2$ if and only if
\begin{equation}\label{6.8}
|\hat{f}(a,b)|=\frac{1}{\sqrt{2\left(3-\cos(2\pi a)-\cos(2\pi b)-\cos(2\pi(a-b))\right)}}. \end{equation}
\end{lemma}
\begin{proof}
We note that for $\tau_{p,q} f([j,k])=f([j,k]+[p,q])$ we have
\begin{equation}\label{6.9}
(\tau_{p,q}f)\hat{}(a,b)=e^{2\pi i (ap+bq)}\hat{f}(a,b)
\end{equation}
from \eqref{6.3}, so
\begin{eqnarray}\label{6.10}
&<F,\tau_{p,q}F>_{\ell^2(\Go)}= \hspace{9cm} \\
&\displaystyle\int_0^1\int_0^1 e^{-2 \pi i (ap+bq)}2(3{-}\cos(2\pi a){-}\cos(2\pi b){-}\cos(2\pi(a-b))) |\hat{f}(a,b)|^2 da db \notag
\end{eqnarray}
by \eqref{6.9} and \eqref{e6.7}.  But the right side of \eqref{6.10} is $\delta(p,q)$ if and only if 
$$2\left(3-\cos(2\pi a)-\cos(2\pi b)-\cos(2\pi(a-b))\right)|\hat{f}(a,b)|^2$$ is identically one, and this is equivalent to \eqref{6.8}
\end{proof}

We are free to choose any phase in \eqref{6.8}, but it is not clear what is to be gained, so we will simply choose $\hat{f}(a,b)$ to be positive.  Note that the only singularity of $\hat{f}$ is near $(0,0)$, where it behaves like $(a^2+b^2)^{-1/2}$, so it is an integrable singularity, but not square integrable.  Thus \eqref{6.4} is everywhere finite and decays like $O \left( (j^2 + k^2)^{-1/2}\right)$.  
Although $f$ is not in $\ell^2(\mathbb{Z}^2)$, we do have $F\in \ell^2(\Go)$.

\begin{theorem}\label{thm6.6}
Let
\begin{equation}\label{6.11}
\tilde{f}([j,k])=\int_0^1\int_0^1 \frac{e^{2\pi i (aj+bk)}}{\sqrt{2\left(3-\cos(2\pi a)-\cos(2\pi b)-\cos(2\pi(a-b))\right)}} da db. \end{equation}
Then $\left\{\Sum_{\mathbb{Z}^2}\tau_{p,q} \tilde{f}([j,k])\psi_{H_{[j,k]}}\right\}$ is an orthonormal basis of $E_6$, and 
\begin{equation} \label{6.12}
\tilde{P_6}F(x)=\Sum_{[p,q]\in\mathbb{Z}^2}\left(\Sum_{y\in\Go}\Sum_{[j,k]\in\mathbb{Z}^2} \tau_{p,q}\tilde{f}([j,k])\psi_{H_{[j,k]}}(y)F(y)\right) \Sum_{[j',k']\in\mathbb{Z}^2} \tau_{p,q}\tilde{f}([j',k']) \psi_{H_{[j',k']}}(x).
\end{equation}
\end{theorem}
\begin{proof}
This is an immediate consequence of Lemma \ref{lem6.5}
\end{proof}


\section{Non-fractafold examples}\label{sec-frSiff}
Theorem \ref{thm-main} an be applied for examples that are not fractafolds. 
We assume that  $\Go=(V_0,E)$ is a finite or  infinite graph which is a union of complete graphs of  3 vertices (it can be said that $\Go$ is a 3-hyper-graph). In principle, we can allow $\Go$ to have  unbounded degrees, as well as loops and multiple edges, but in this section we will keep everything simple and  assume that $\Go$ is a regular graph.   As before, each of these complete 3-graphs we call a cell, or a 0-cell, of $\Go$. We denote the discrete \Lp\ on $\Go$ by $\D_\Go$. 
\begin{figure}[htb]\centering
\begin{picture}(120, 102)(57, 0)
\multiput(0, 0)(0,50.4){2}{\multiput(0, 0)(28.6,0){8}{\includegraphics[height=26pt,width=30pt]{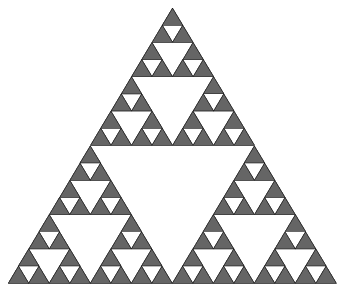}}}
\multiput(-14.3, 25.2)(0,50.4){2}{\multiput(0, 0)(28.6,0){8}{\includegraphics[height=26pt,width=30pt]{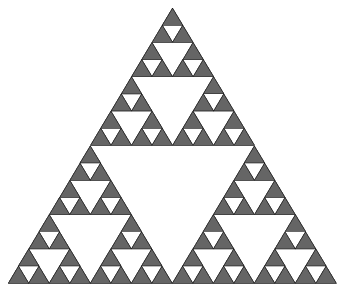}}}
\end{picture}
\caption{A part of the  periodic triangular lattice \frSiff. This fractal field is not a fractafold. }\label{fig-t-lattice}
\end{figure}
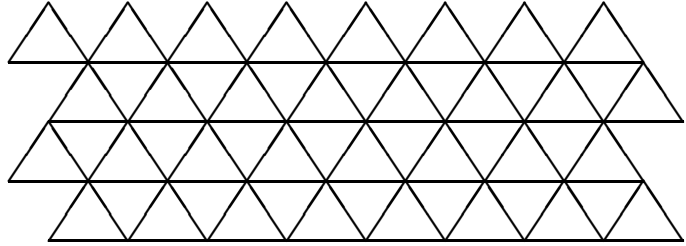
\begin{figure}[htb]\centering
\begin{picture}(120, 90)(57, 0)\thicklines
\multiput(0, 0)(0,45){2}{\multiput(0, 0)(30,0){8}{
\setlength{\unitlength}{1pt}
\put(0,0){\line(1,0){30}}
\put(0,0){\line(2,3){15}}
\put(30,0){\line(-2,3){15}}
}}
\multiput(-15, 22.5)(0,45){2}{\multiput(0, 0)(30,0){8}{
\setlength{\unitlength}{1pt}
\put(0,0){\line(1,0){30}}
\put(0,0){\line(2,3){15}}
\put(30,0){\line(-2,3){15}}
}}
\end{picture}
\caption{A part of the infinite  triangular lattice, the $\Go$ graph for the fractal field in Figure~\ref{fig-t-lattice}. }\label{fig-t-lattice-Go}
\end{figure}
We define a \emph{\frSiff}\ $\fF$ by replacing each cell of $\Go$ by a copy of ${SG}$. These copies we call  cells, or 0-cells, of $\fF$. 
Naturally, the corners of  the copies of the \Sig\ ${SG}$ are identified with the vertices of $\Go$. See \cite{HK2} for fractal fields, not necessarily finitely ramified. 
Since the pairwise intersections of the 
cells of $\fF$ are finite, we can consider the natural measure on $\fF$, which we also  denote $\mu$. Furthermore, since $\DS$ is a local operator, we can define a local  \Lp\ $\D$ on $\fF$, in the same way as explained in \cite{St03} (this means that the sum of normal derivatives is zero at every junction points). 
One can see that most of our results can be easily generalized for the \frSiff s. For instance, Theorem~\ref{thm-main} is essentially still valid. 
One change to be made is that on the graph $\G$ we have to consider the 
probabilistic \Lp\ (which is explained in \cite{MT03,Sh}), and multiply it by 4 to 
align with the normalization of the \Lp\ on the \Sig. 

In the example shown in Figure~\ref{fig-t-lattice-Go}, the spectrum  on $\Go$ is $[0,8]$ for the adjacency matrix \Lp, and  the spectrum is 
$[0,4/3]$ for the probabilistic \Lp. Thus $\Sigma_0=[0,\frac{16}3]$. In this particular case the spectrum is  absolutely continuous by the classical theory (see 
\cite[\arti]{Kuchment91, Kuchment93,  Kuchment05, KuchmentPost, KuchmentVainberg} 
for a sample of relevant recent results on periodic \Lp s). Combining the methods described in this paper, we obtain the following proposition (see also Figure~\ref{fig-frSiff-s}).

\begin{figure}[hbt]\centering
\def\q{\thicklines\qbezier(-2,-1)(0,0)(2,1)\qbezier(-2,1)(0,0)(2,-1)}
\def\qq{\put(0,0){\q}\put(0,2){\q}\put(0,4){\q}\put(0,6){\q}}
\def\qqq{\put(0,2){\qq}\put(0,15){\q}\put(0,20){\qq}}
\begin{picture}(400, 110)(0, 0)
\put(00,00){\linethickness{1.75pt}\line(0,1){80}}
\def\q{\setlength{\unitlength}{1pt}\thicklines\qbezier(-1,-3)(0,0)(1,3)\qbezier(-1,3)(0,0)(1,-3)}
\put(-7,0){\SMALL0}\put(-15,45){$\Sigma_0$}\put(-11,77.5){\SMALL$\frac{16}3$}\put(-7,90){\SMALL6}
\put(0,0){\vector(1,0){400}}\put(0,0){\vector(0,1){100}}
\multiput(0,90)(2.5,0){160}{\line(1,0){1}}
\multiput(90,0)(0,2.5){36}{\line(0,1){1}}\put(90,0){\q}
\put(150,0){\q}\put(200,0){\q}\put(380,0){\q}
\multiput(308,0)(0,2.5){36}{\line(0,1){1}}\put(308,0){\q}
\put(0,0){\setlength{\unitlength}{15pt}\qbezier(0,0)(5,12.5)(10,0)}
\put(200,0){\setlength{\unitlength}{15pt}\qbezier(0,0)(6,12.5)(12,0)}
\put(00,00){\linethickness{1.75pt}\line(1,0){52}}
\put(150,00){\linethickness{1.75pt}\line(-1,0){52}}
\put(200,00){\linethickness{1.75pt}\line(1,0){62}}
\put(380,00){\linethickness{1.75pt}\line(-1,0){62}}
\end{picture}
\caption{Computation of the spectrum on the  triangular lattice \frSiff.}\label{fig-frSiff-s}
\end{figure}


\BB{proposition}{thm-last}{The \Lp\ on the  periodic triangular lattice \frSiff\ consists of absolutely continuous spectrum and pure point spectrum. The absolutely continuous spectrum is $\fR^{-1}[0,\frac{16}3]$. The pure point spectrum consists of two infinite series of eigenvalues of infinite multiplicity. The series 
$5\fR^{-1}\{3\}\subsetneq\fR^{-1}\{6\}$ consists of isolated eigenvalues, and the series  $5\fR^{-1}\{5\}=\fR^{-1}\{0\}\backslash\{0\}$ is at the gap edges of the \mbox{a.c.} spectrum. The eigenfunction with compact support are complete in the \mbox{p.p.} spectrum. The spectral resolution is given by \eqref{e8}.}

It is straightforward to generalize such a result for other \frSiff s (see, in particular,  Remark~\ref{rem-K}).


\end{document}